\documentclass[ envcountsect]{svjour3}
\usepackage[lmargin=1in,tmargin=1in,rmargin=1in,bmargin=1in]{geometry}
\usepackage{amsmath,amssymb,latexsym, amscd,epsfig,epstopdf,enumerate}
\usepackage{graphicx,subfig,pdfpages}

\usepackage[utf8x]{inputenc}
\usepackage{amssymb,pdfsync,epstopdf,verbatim,gensymb}
\usepackage{nccmath}
\usepackage{textcomp}
\usepackage{amsmath}
\usepackage{float}
\usepackage{cite}
\usepackage{amsfonts}
\usepackage{gensymb}
\usepackage{color}
\usepackage{hyperref}
\usepackage{multirow}

\usepackage{caption}
\graphicspath{{./geometry_pics/}}
\begin{document}
\title{Energy Minimization Problem related to certain Liquid Crystal Models with respect to the Position, Orientation and Size of an Obstacle having Dihedral Symmetry 
} 
\author{Anisa M. H. Chorwadwala 
}
\institute{ Anisa M. H. Chorwadwala \at 
		Indian Institute of Science Education and Research Pune,\\
        Dr. Homi Bhabha Road, Pashan, Pune 411008, India, \\
        Tel: +91(20)25908218, \\
		anisa@iiserpune.ac.in
}
\titlerunning{Obstacle Position, Orientation and Size}
\maketitle
\smartqed

\begin{abstract} 
  In this paper, we deal with an obstacle placement problem inside a disk, that can be formulated as an energy minimization problem with respect to the rotations of the obstacle about its center, with respect to the translations of the obstacle within a planar disk and also with respect to the expansion or contraction of the obstacle inside the disk. We also include its sensitivity with respect to the boundary data.
  
  We consider the inhomogeneous Dirichlet Boundary Value Problem (\ref{laplace_equation_lc}) on $\Omega = B \setminus P$ for $g=M$, a nonzero constant. Here, the planar obstacle $P$ is invariant under the action of a dihedral group $\mathbb{D}_n$, $n \in \mathbb{N}$, $n \geq 3$, and $B$ is a disk in $\mathbb{R}^2$ containing $P$. We make assumptions (0) to (iv) given in section \ref{assumptions} on $P$ and $B$. The obstacle $P$ satisfying these conditions becomes a star-shaped domain with respect to its center $\underline{o}$. Examples of such obstacles include regular polygons with smoothened corners, ellipses, among other star-shaped domains satisfying conditions (0) to (iv) mentioned in section \ref{assumptions}. Please see figure 1 on page 1114 of \cite{elsoufi_kiwan} for more examples of admissible obstacles. In this setting, we investigate the extremal configurations of the obstacle $P$ with respect to the disk $B$, for the energy functional $E$ given by (\ref{energy_functional}) associated with the boundary value problem (\ref{laplace_equation_lc}). 
  
The Dirichlet energy (\ref{energy_functional}) is the one-constant Oseen-Frank energy of a uniaxial planar nematic liquid crystal. The energy minimizing or maximizing nematic profiles correspond to extremals of (\ref{energy_functional}) and are hence, described by harmonic functions, subject to suitable boundary conditions. The class of problems considered here sheds some light into optimal locations of holes with a certain symmetry inside a nematic-filled disc, so as to optimise the corresponding one-constant nematic Oseen-Frank energy.
 \end{abstract}
 


\section{Introduction}
 %

In this paper, we consider the family of domains 
consisting of doubly connected planar domains with specified boundary conditions on the inner and outer boundary component. We will consider a disk minus an obstacle with dihedral symmetry. In this case, 
the corresponding nematic profile, modelled as an energy minimizer, depends not only on the location of the center of the inner obstacle relative to the outer disk but also on the orientation of the obstacle w.r.t. the diameter of the outer disk passing through the center of the obstacle. It, of course, does depend on the size of the obstacle and the boundary data too. We aim to find the optimal orientation of the inner obstacle about its fixed center inside the disk. We also study the behaviour of energy functional w.r.t. the translation of the obstacle within the disk. The behaviour of the energy functional w.r.t. the expansion or contraction of the obstacle inside the disk is considered too. We also include its sensitivity w.r.t. the boundary data of (\ref{laplace_equation_lc}). 

\subsection{Motivation} For a motivation of placing an obstacle with dihedral symmetry inside a disk one could refer to \cite{Anisa-Souvik}. 
We include it here again for completeness. Such problems apply, for example, to the designing of some musical instruments, where one usually has a symmetric structure of the obstacle with respect to the domain. The energy minimization and the Dirichlet eigenvalue optimization problem for doubly connected, constant area, planar domains where the domain and the obstacle both have a rectifiable boundary curve (with no a priori symmetry) is an open problem. Solving this problem for the case where both the outer and the inner boundary curves are circular is a step towards solving this challenging problem. Both the outer boundary and the inner boundary having dihedral symmetry can be thought of as the second best symmetry that one could have. A version of this was addressed in \cite{elsoufi_kiwan}. But for them the obstacle and the domain were always concentric. We wanted to study the non-concentric version of the problem they considered. Therefore, we first consider the case where the domain is circular and the obstacle has a dihedral symmetry. In future we hope to understand the non-concentric version of the problem in \cite{elsoufi_kiwan} where both the boundaries have a dihedral symmetry. We would like to bring to the attention of the reader the second, third and fourth paragraphs on page 1113 of \cite{elsoufi_kiwan} to highlight what its authors feel about addressing the problem in a more general setting.

To see more motivation for considering an obstacle with a dihedral symmetry, it's a good idea to recall the proof of the classical isoperimetric problem for planar domains. This proof makes use of the standard approximation arguments and the isoperimetric theorem for constant area polygons with $n$ number of sides. The isoperimetric theorem for polygons states that ``There exists a perimeter minimizer among all polygons with $n$ number of sides enclosing a fixed area $A$ and it is the regular polygon 
with $n$ number of sides having area $A$". 

This is one reason why understanding an optimization problem for polygons helps in approaching the problem in the general setting. And, studying the case for regular polygons is a step towards it. In our current work, we assume a smooth boundary for the obstacle though. The case where the boundary may not be smooth (polygons, for example), follows by the arguments similar to \cite{aithal_acushla,aithal_raut}. In our paper, we deal with a larger class of obstacles in the sense that we do not assume that the boundary curve of the obstacle is piecewise linear. Please note that the regular polygons with smoothed corners, ellipses, among other star-shaped domains satisfying conditions (0) to (iv) mentioned in section \ref{assumptions} are admissible candidates for the obstacle. Please see figure 1 on page 1114 of \cite{elsoufi_kiwan} for more examples of admissible obstacles. 

For PDEs, the reduction of variables through dimensional analysis, is connected to the reduction from invariance under groups of symmetries. Please see \cite{BGKS}. 

Another motivation for placing an obstacle with dihedral symmetry inside a disk, could be that we have a disk filled with liquid crystals (LC) and the obstacle could be a foreign inclusion. We could model inclusions of different shapes and symmetries inside a LC-filled disk. The problem with homogeneous boundary conditions is not interesting from the liquid crystals point of view. Please see \cite{Canevari-Majumdar-Spicer}, \cite{Majumdar-Lewis}, \cite{Majumdar-Wang-Canevari} and \cite{LAHM}. 
Although, in the current paper, the connection to liquid crystals is only tangential, we aim to postulate problems in liquid crystal modelling where these results can be of value. They can effectively be interpreted as boundary value problems for a two dimensional nematic director field inside a disk with certain kinds of obstacles and fixed constant boundary conditions. 

The Dirichlet energy, $\int \|\nabla \theta\|^2 dA$, 
is the one-constant Oseen-Frank energy of a uniaxial planar nematic liquid crystal. The state of the uniaxial nematic is described by a two-dimensional vector field, $n = (\cos \theta, \sin \theta)$ such that $n$ describes the locally preferred direction of averaged molecular alignment at every point of space, subject to suitably prescribed boundary conditions. The energy minimizing or maximizing nematic profiles correspond to extremals of the energy functional and are hence, described by harmonic functions, $\theta$, subject to suitable boundary conditions. The class of problems considered in this paper sheds some light into optimal locations of holes with a certain symmetry inside a nematic-filled disc, so as to optimise the corresponding one-constant nematic Oseen-Frank energy.


\subsection{Problem formulation}\label{assumptions}In this paper, we consider the following inhomogeneous Dirichlet Boundary Value Problem:
\begin{equation}\label{laplace_equation_lc}
\begin{aligned}
\Delta u &=0 & \mbox{ in } & \Omega:=B \setminus P,\\
&u =0 & \mbox{ on } & \partial P,\\
&u =g 
& \mbox{ on } & \partial B.
\end{aligned}
\end{equation}
For us $g=M$, a nonzero constant. Here, we take the same family of admissible domains $\Omega$ as in \cite{Anisa-Souvik}. That is, we consider the case where the planar obstacle $P$ is invariant under the action of a dihedral group $\mathbb{D}_n$ $n \geq 3$, $n \in \mathbb{N}$
. It follows that the axes of symmetry of $P$ intersect in a unique point in the interior of $P$. We call this point the center of $P$ and denote it by $\underline{o}$. Let $B$ be a disk in $\mathbb{R}^2$ containing $\underline{o}$ away from its center. We place the obstacle $P$ centered at the fixed point $\underline{o}$ inside $B$. 
That is, the centers of $P$ and $B$ are distinct. 
The disk $B$ obviously is invariant under the action of dihedral groups $\mathbb{D}_n$, for each natural number $n ≥ 3$. Therefore, in our case, we have the following: (0) the volume constraint on $P$ and $B$ both, (i) invariance of both $B$ and $P$ under the action of the same dihedral group, 
(ii) $B$ and $P$ need not be concentric, (iii) smoothness condition on both the boundaries, 
(iv) the monotonicity condition on the boundary $\partial P$ of the obstacle $P$ as in \cite{Anisa-Souvik} and \cite{elsoufi_kiwan}. That is, the distance $d(\underline{o}, x)$, between the center $\underline{o}$ of $P$ and the point $x$ on the boundary $\partial P$ of $P$, is monotonic as a function of the argument $\phi$ in a sector delimited by two consecutive axes of symmetry of $P$. Without loss of generality, we can assume that the centre of $P$ is $\underline{o}:=(0,0)$. 
We recall a monotonicity condition on the boundary of the disk $B$ derived in Lemma 4.1 
 of \cite{Anisa-Souvik}. Therefore, for us condition (v) of \cite{elsoufi_kiwan} for $B$ is replaced by the statement of Lemma 4.1 
 of \cite{Anisa-Souvik}. 
Theorem 8.14 on page 188 of \cite{Gilberg-Trudinger} states that the boundary value problem (\ref{laplace_equation_lc}) admits a unique solution in $\mathcal{C}^\infty(\overline{\Omega})$. 

In this setting, we investigate the extremal configurations of the obstacle $P$ with respect to the disk $B$, for the energy functional 
\begin{equation}\label{energy_functional}
E(\Omega):= \int_\Omega \|\nabla u\|^2\,dx
\end{equation}
 associated with (\ref{laplace_equation_lc}), by rotating $P$ inside $B$, about the fixed center $\underline{o}$ of $P$. We also study the behaviour of $E$ w.r.t. the translation of $P$ within $B$, and w.r.t. the expansion/contraction of the obstacle inside the disk $B$. We also include its sensitivity w.r.t. the boundary data $g=M$.
  
 The reviewers of the paper pointed out to us that this problem can be interpreted as an optimal placement of a cavity or inclusion of specified shape in the domain of a planar disk, where the state field $u(x)$ satisfies the harmonic equation and Dirichlet boundary conditions and that more applications can be related to thermal or electrical fields control. The posed problem (\ref{laplace_equation_lc}) can be interpreted in terms of a steady state temperature field $u =T(x)$ in a disk heated by a rectangular or polygonal channel with a specified heating temperature on its boundary. We continue to use the term ``obstacle" following \cite{elsoufi_kiwan} and \cite{harrel_kurata}.

We follow the same line of ideas as in \cite{Anisa-Souvik} and \cite{harrel_kurata}. A novelty, as compared to \cite{Anisa-Souvik} and \cite{harrel_kurata}, is going to be the computation of (a) the shape derivative of the solution of a boundary value problem with inhomogeneous boundary data, (b) the Eulerian derivative of the energy functional w.r.t. the variations of the domain under the action of certain vector fields, (c) the study of the behaviour of the energy functional w.r.t. the expansion/contraction of the obstacle $P$ inside the disk $B$, and (d) its sensitivity to the boundary data. We characterize the global maximizers/minimisers w.r.t. both rotations and translations of the obstacle within the disk. We do indicate what happens when the size of the obstacle approaches zero. 

The reviewers of this paper brought to our attention the extensive research conducted in last two decades on heat conduction and thermo-elasticity with shape sensitivity analysis of temperature field and state functionals for varying external boundaries and interfaces. In fact, the steady state temperature field satisfies the same harmonic equation and mixed boundary value conditions. Please see \cite{DM1}, \cite{DM2}. According to them, the derivations presented in the present paper are specific examples of more general formulae for arbitrary state functionals, also for the cases of rotation and translation of inclusions or cavities.

\subsection{Presentation of the paper} 
In order to identify the various different configurations in the family of domains under consideration we recall, in section \ref{sec:prelim}, a few important definitions from \cite{Anisa-Souvik} through a sequence of illustrative images. 
In Section \ref{stmnt}, we state our main theorems. Theorem \ref{max_min} describes the extremal configurations for the energy functional associated to (\ref{laplace_equation_lc}) over the family of admissible domains. This theorem also characterises the maximising and the minimising configurations for it. It is worth emphasising here that Theorem \ref{max_min} describes the behaviour of the energy functional w.r.t. the rotations of the obstacle about its fixed center.
In Theorem \ref{global} we characterize the global maximizing and the minimizing configurations for $E$ w.r.t. the rotations of the obstacle about its center, where the center is allowed to translate within the disk. 
We state a result about the critical points for $n$ odd too. Theorem \ref{expansion} states that the energy functional strictly decreases as the obstacle contracts under a scaling factor and that it approaches its minimum value as the obstacle shrinks to a point. 
In section \ref{scal}, we carry out the shape calculus analysis for the boundary value problem at hand and derive an expression for the Eulerian derivative of the energy functional. We observe that the energy functional increases as the value of $|M|$ increases. 
In section \ref{proof}, we give a proof of the extremal configurations for obstacles with even order dihedral symmetry. We prove the result for $n$ odd stated in section \ref{stmnt}. In section \ref{secGlobal}, we provide a proof for the global extremal configurations. 
In section \ref{secExpansion}, we give a proof of Theorem \ref{expansion}. 
We then indicate what happens when the size of the obstacle approaches zero. 
Section \ref{sec:num_results} presents some numerical results that validate the the extremal configurations obtained and also justify the conjectures formulated. 
 In section \ref{Conc}, we summarize the important conclusions of the paper. 
\section{
The ON and OFF configurations}\label{sec:prelim}
Let $n$ be a positive integer, $n \geq 3$. Let $P$ be a compact and simply connected subset of $\mathbb{R}^2$ satisfying conditions (0) to (iv) given in section \ref{assumptions}. 
We direct the reader to section 2 of \cite{Anisa-Souvik} for definitions of the 
incircle $C_1(P)$ and the circumcircle $C_2(P)$ of the obstacle $P$, vertices of $P$, opposite vertices of $P$, inner vertices and outer vertices of $P$. 
When there is no scope of confusion, we will use the notations $C_1$ and $C_2$ for  $C_1(P)$ and $C_2(P)$, respectively. Let $B$ be an open disk in $\mathbb{R}^2$ of radius $r_1$ such that $B \supset cl(conv(C_2(P)))$. Let $\Omega := B \setminus P$. 
Please refer to section 3.2 of \cite{Anisa-Souvik} for the definitions of ON and OFF configurations of the obstacle w.r.t. the disk $B$. We only retain the pictures here as a quick recap to illustrate those definitions. Without loss of generality we assume that the center $\underline{o}$ of the obstacle $P$ is at the origin $(0,0)$ of $\mathbb{R}^2$.
\begin{figure}[ht]
\centering
\subfloat[$\mathbb{D}_4$ symmetry]{\includegraphics[width=0.15\textwidth]{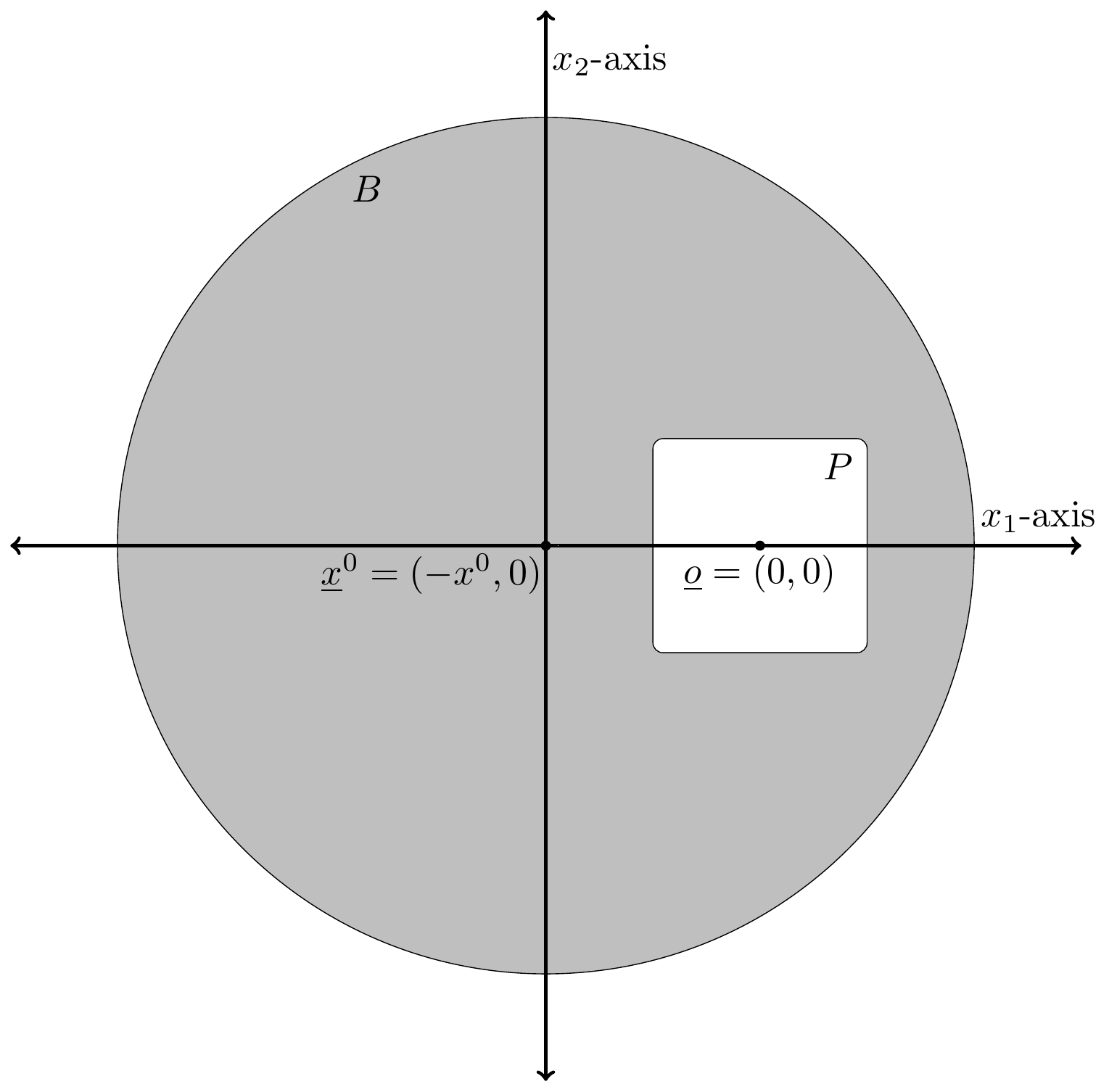}    }\hspace{10mm}
\subfloat[$\mathbb{D}_4$ symmetry]{\includegraphics[width=0.15\textwidth]{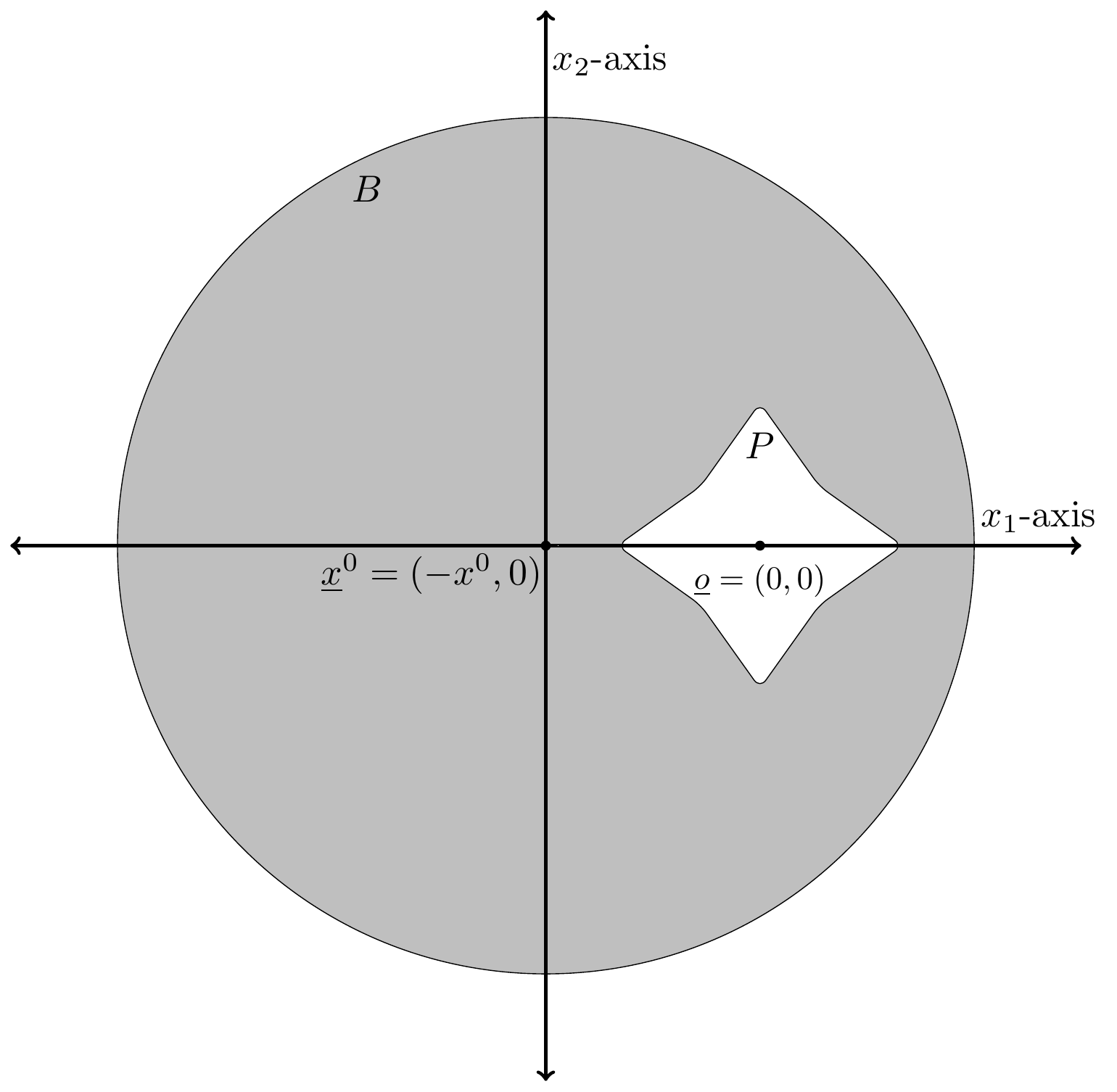}    }\hspace{10mm}
\subfloat[$\mathbb{D}_3$ symmetry]{\includegraphics[width=0.15\textwidth]{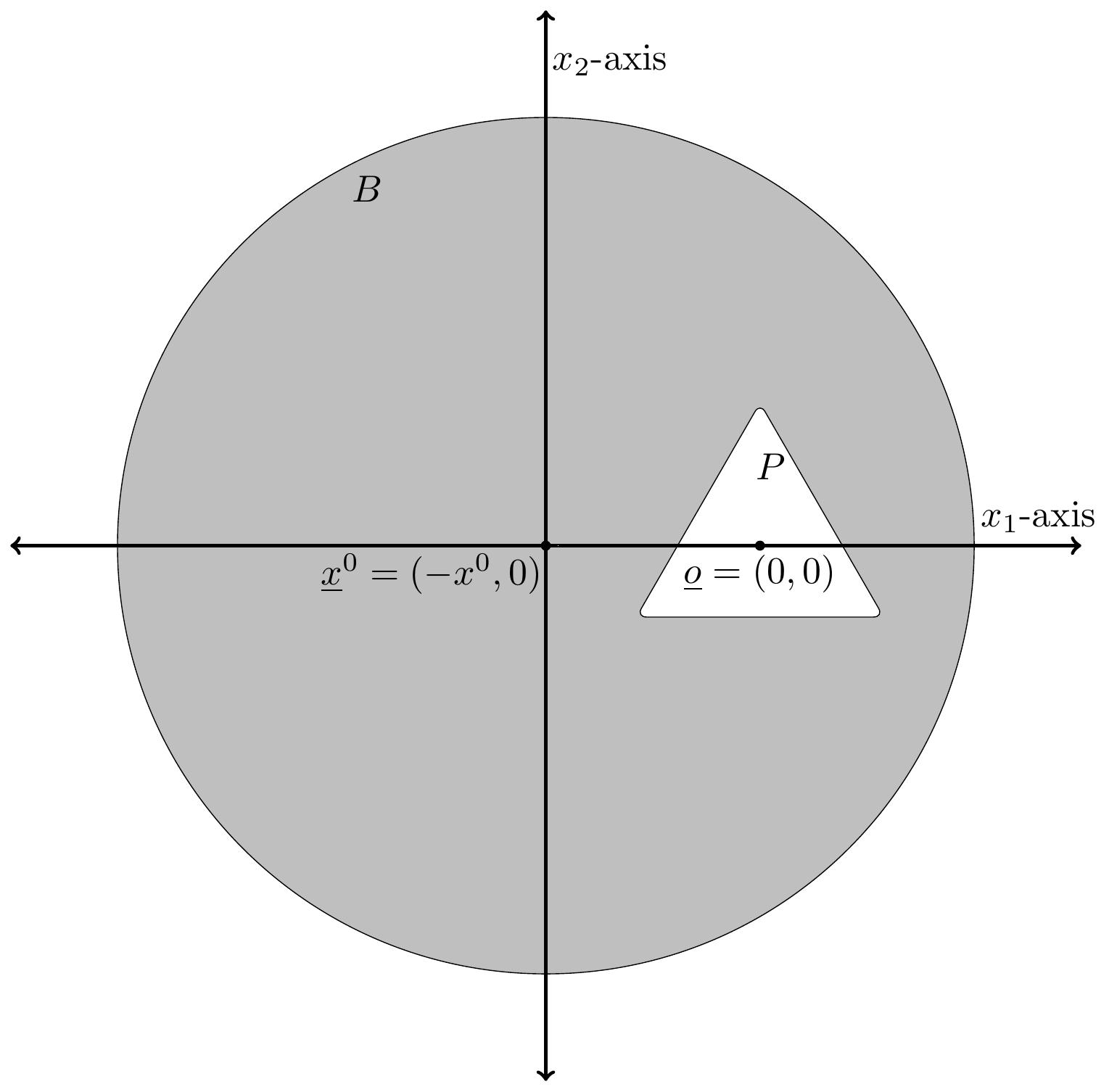}    } \hspace{10mm}
\subfloat[$\mathbb{D}_5$ symmetry]{\includegraphics[width=0.15\textwidth]{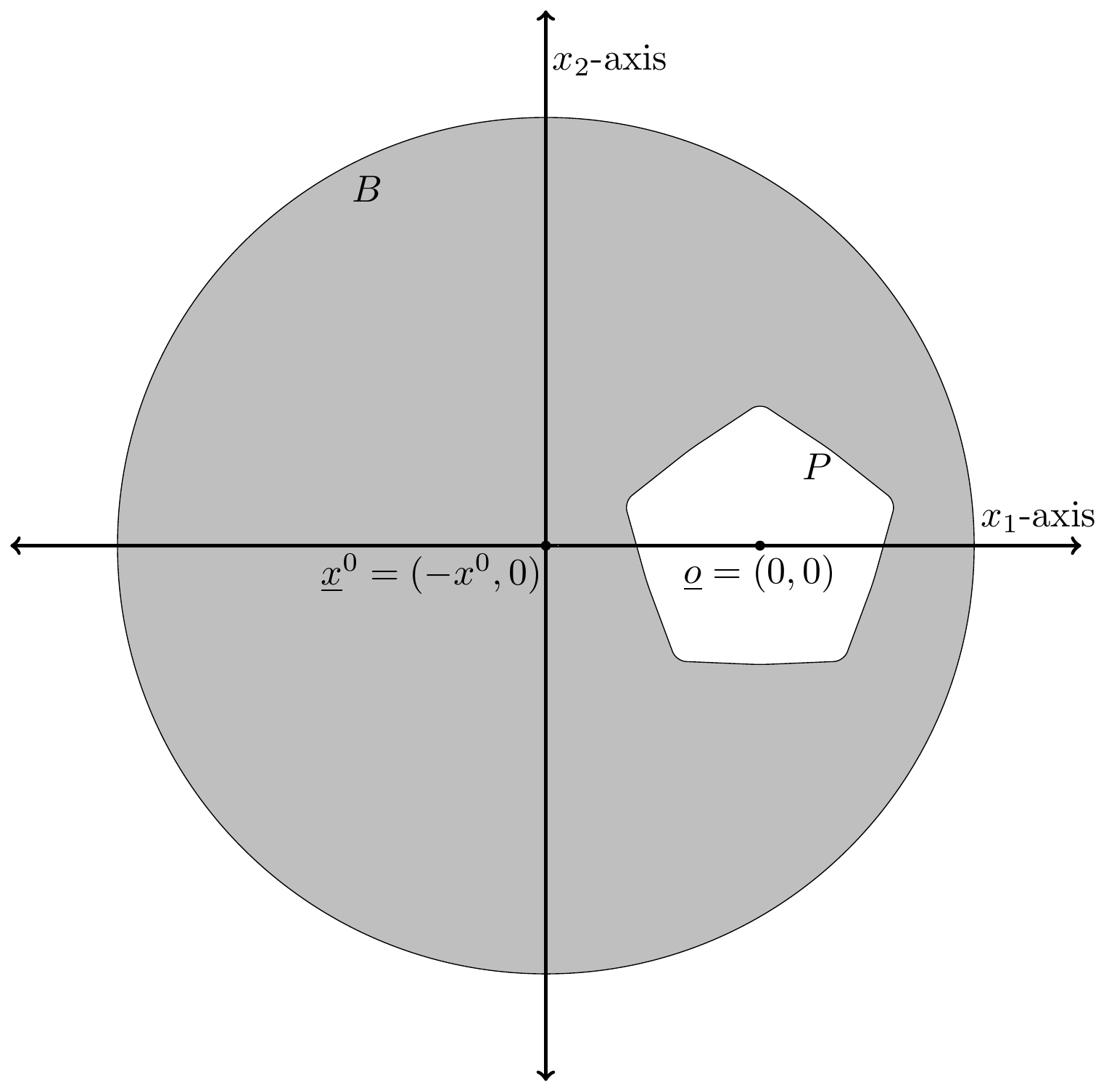}    }
\caption{Obstacles having $\mathbb{D}_n$ symmetry
}\label{fig:domain_odd}
\end{figure}
\begin{figure}[H]\centering
\subfloat[]{
\includegraphics[width=0.25\textwidth]{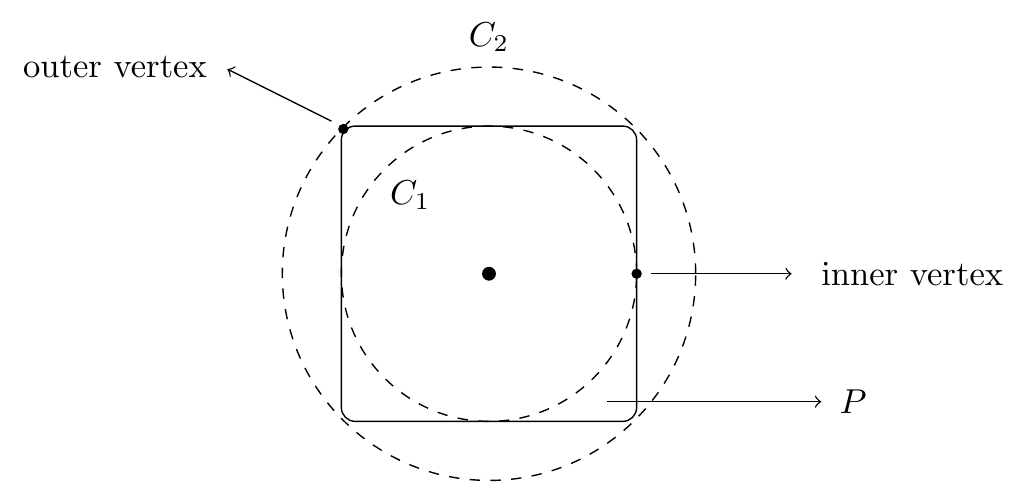}}
\subfloat[]{
\includegraphics[width=0.25\textwidth]{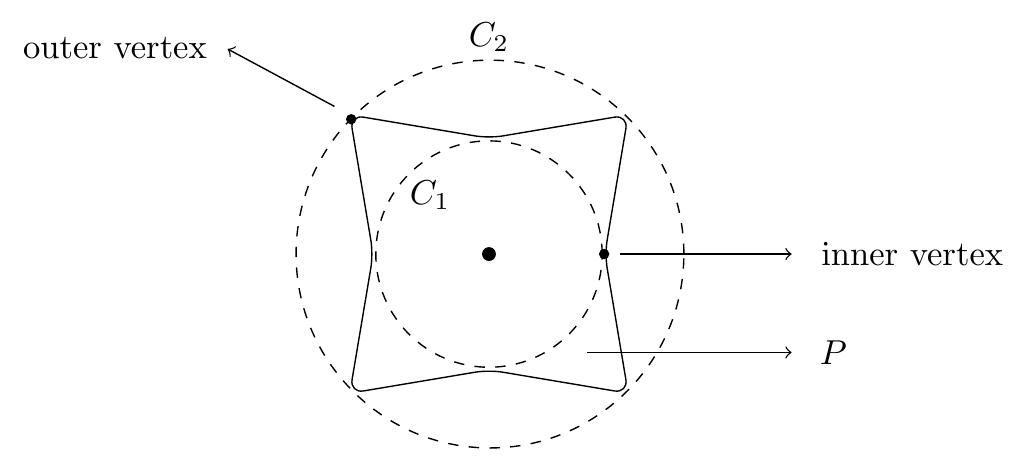}}
\subfloat[]{
\includegraphics[width=0.25\textwidth]{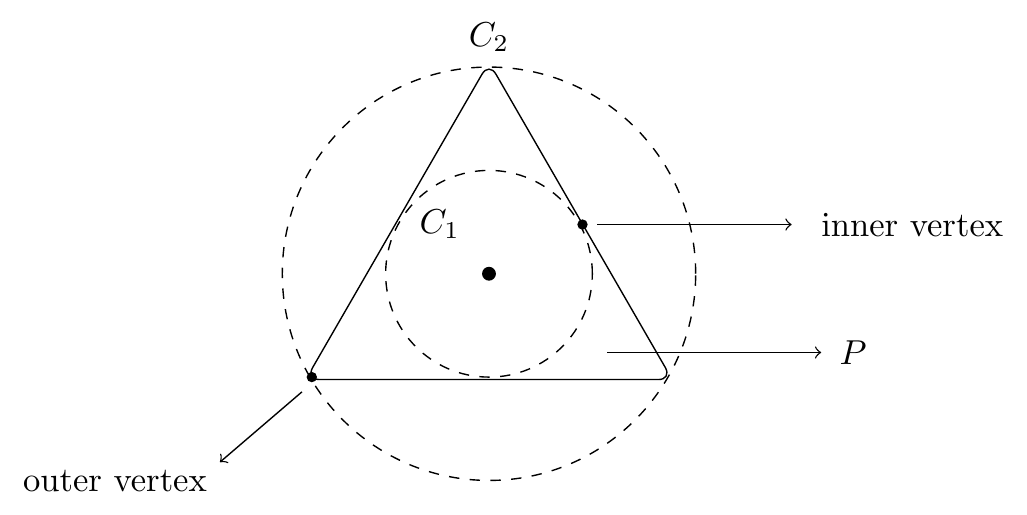}}
\subfloat[]{
\includegraphics[width=0.25\textwidth]{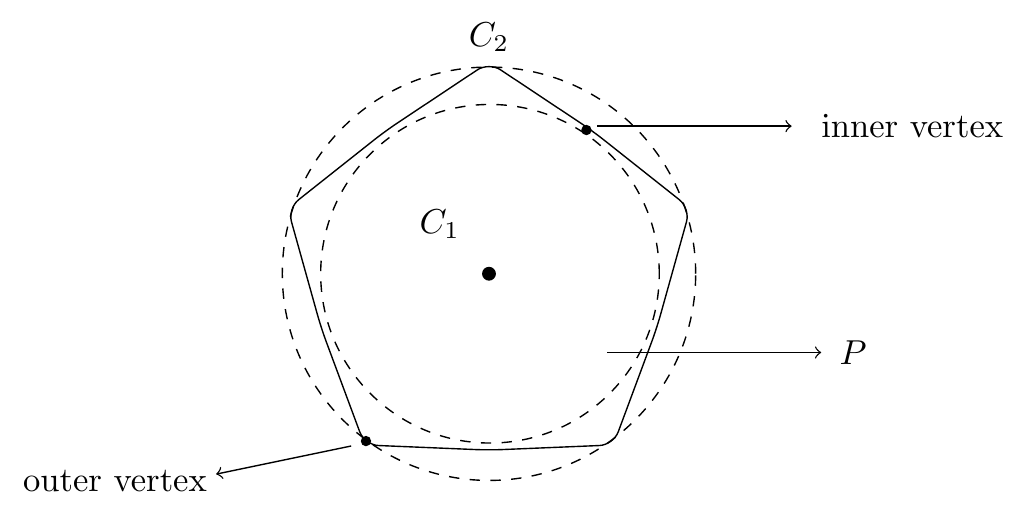}}
\label{Vertex}
\caption{Vertices of $P$ }
\end{figure}
\begin{figure}[H]\centering
\subfloat[OFF 
]{
\label{Off position}
\includegraphics[width=0.1\textwidth]{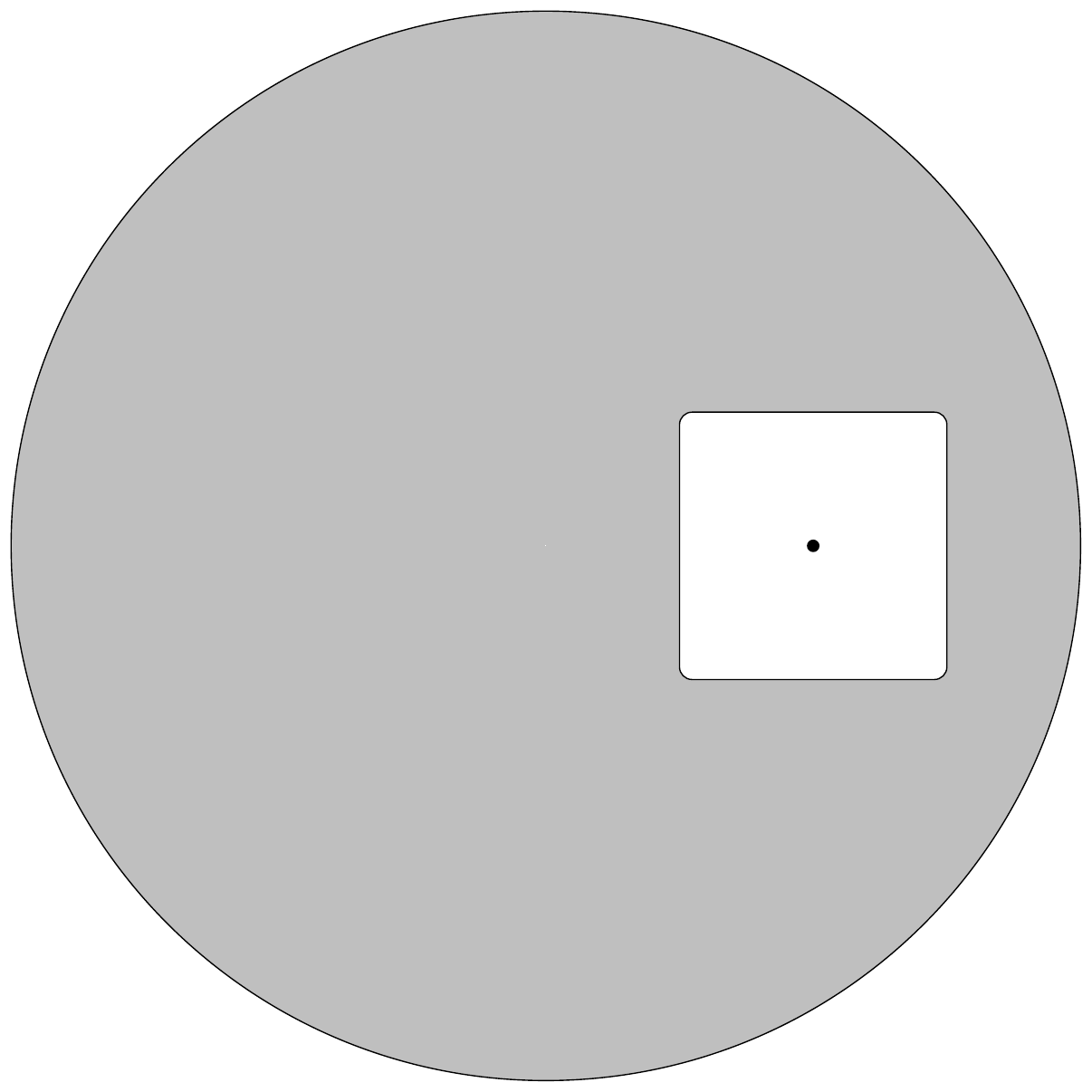}}
\hspace{8mm}
\subfloat[ON 
]{
\label{On position}
\includegraphics[width=0.1\textwidth]{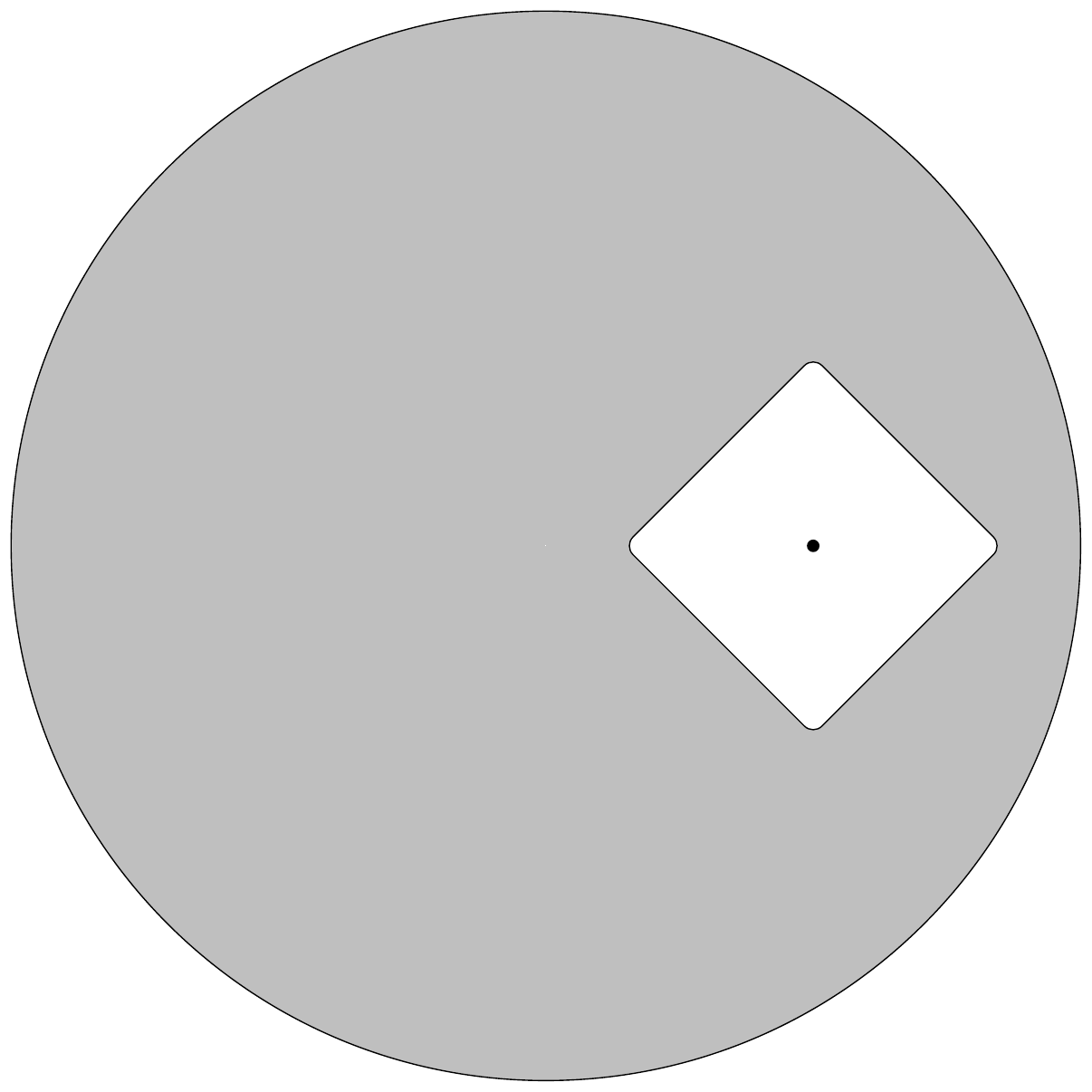}}
\hspace{10mm}
\subfloat[OFF 
]{
\label{Off position2}
\includegraphics[width=0.1\textwidth]{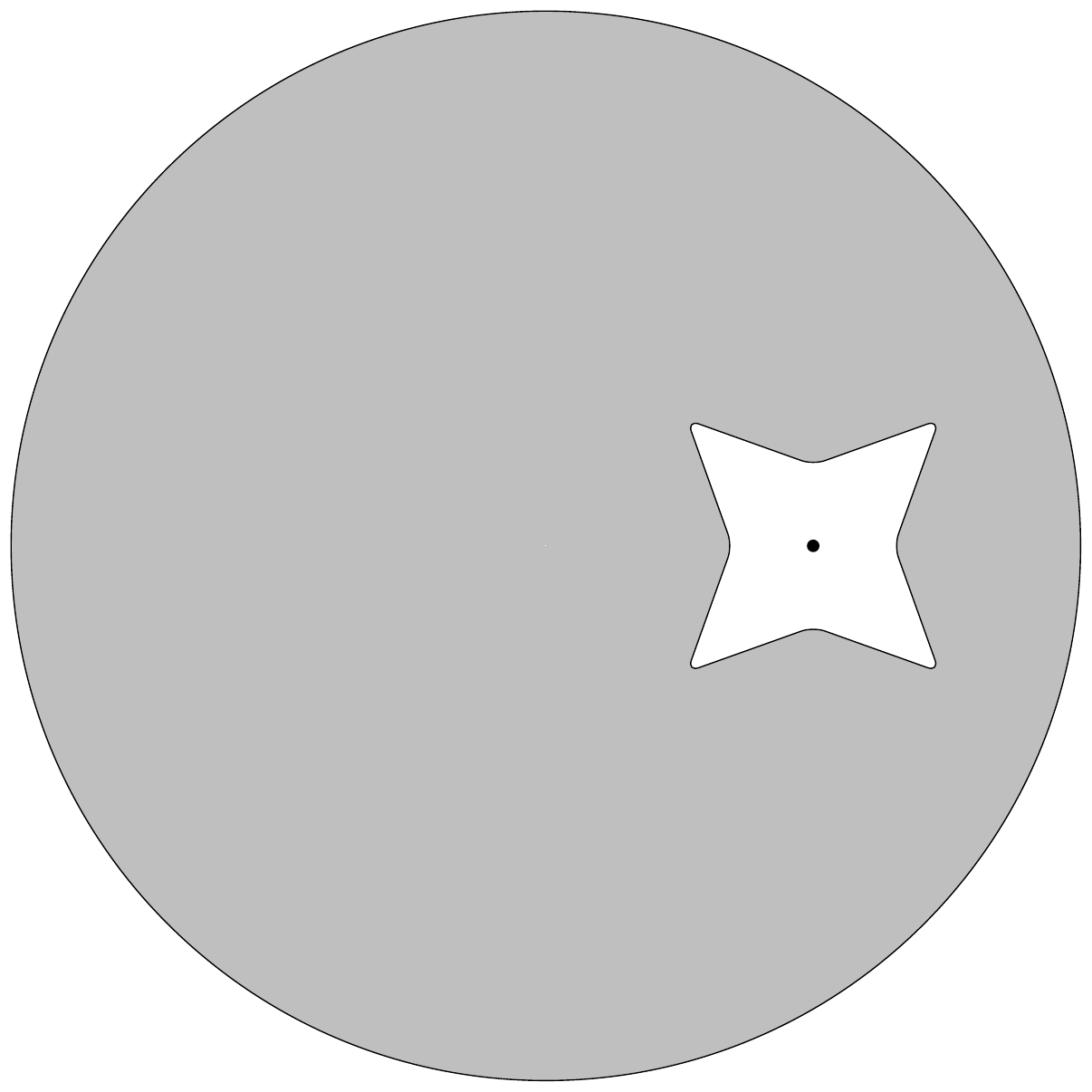}}
\hspace{8mm}
\subfloat[ON 
]{
\label{On position2}
\includegraphics[width=0.1\textwidth]{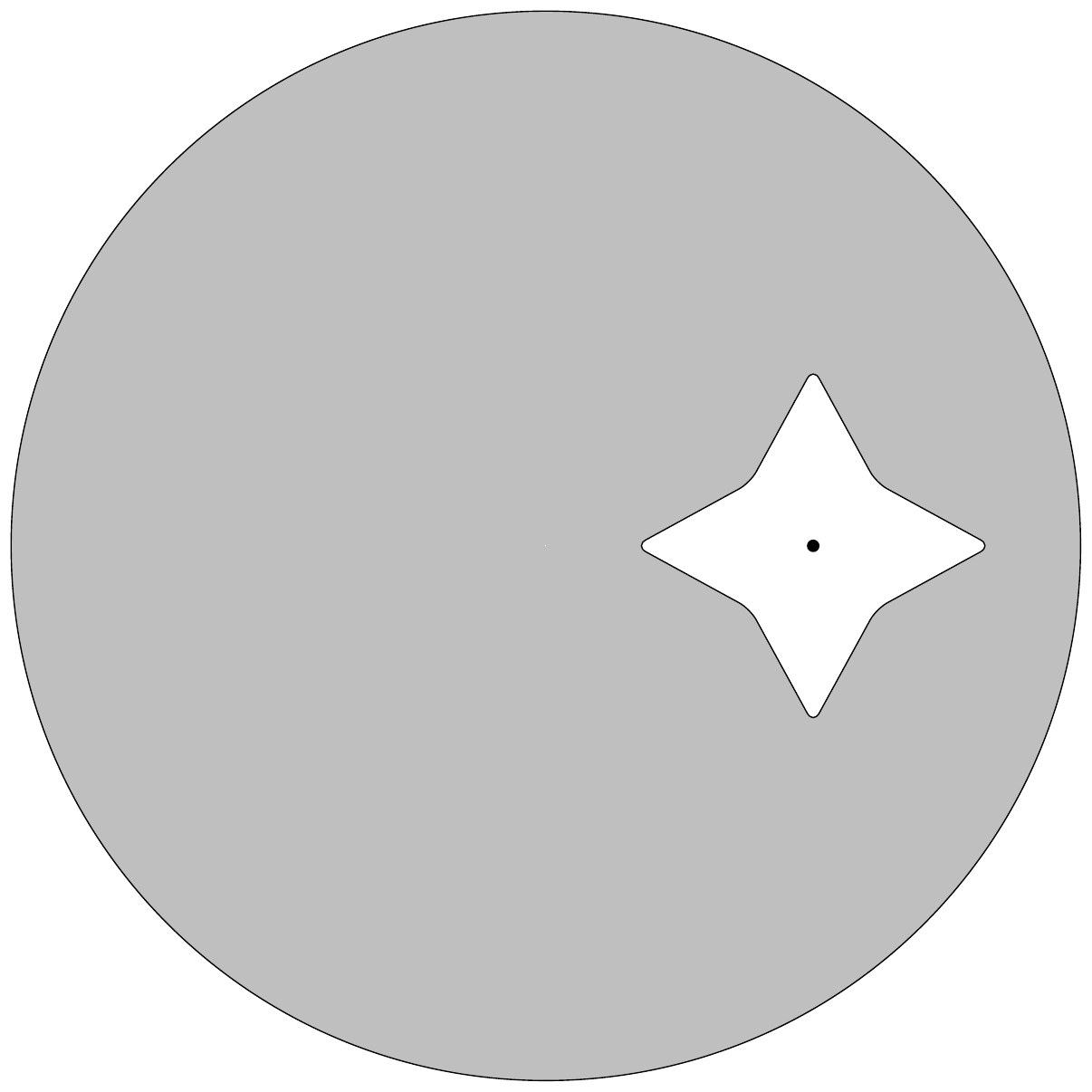}}
\hspace{10mm}
\subfloat[OFF 
]{
\label{Off position_ellipse}
\includegraphics[width=0.1\textwidth]{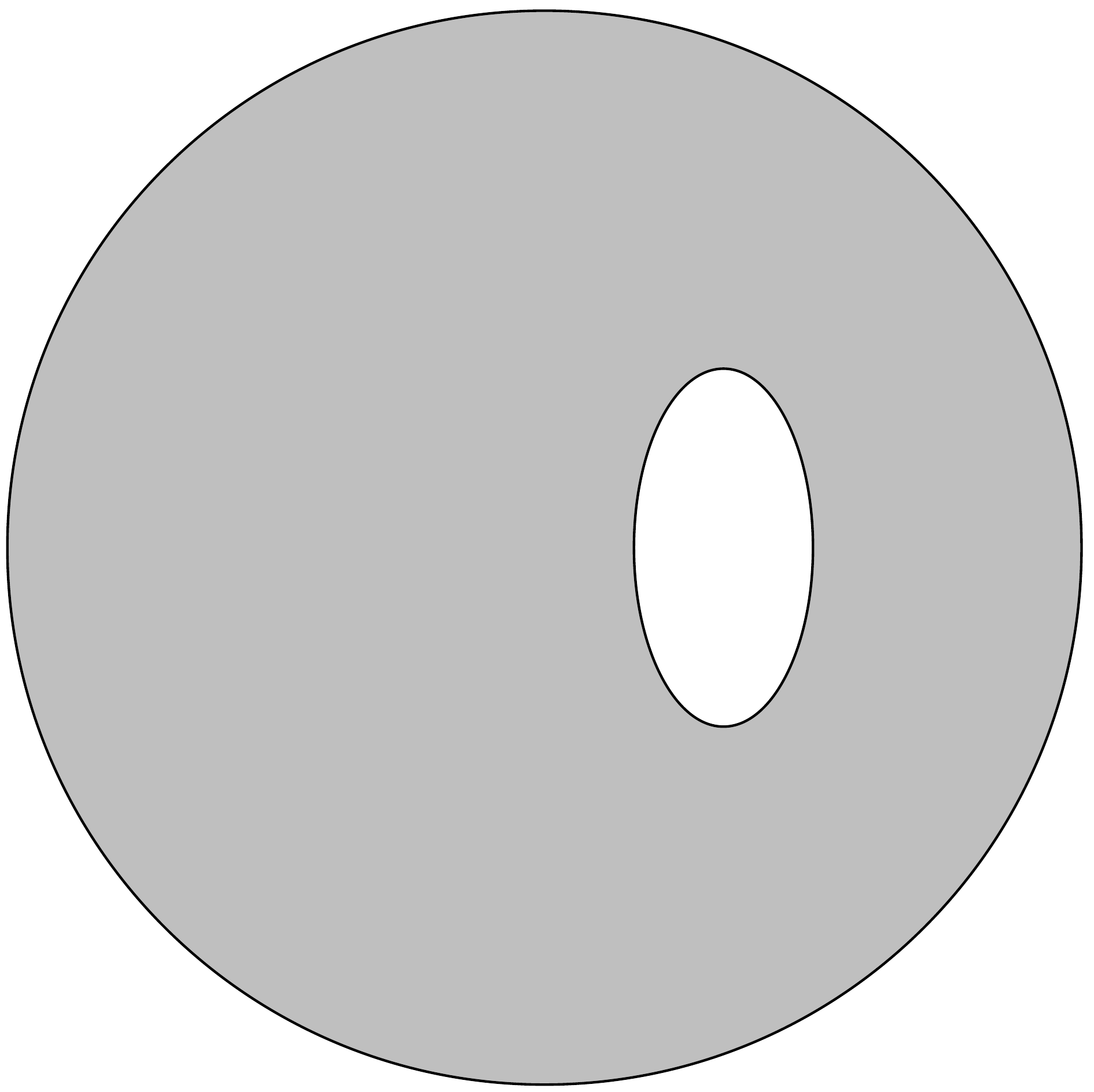}}
\hspace{8mm}
\subfloat[ON 
]{
\label{On position_ellipse}
\includegraphics[width=0.1\textwidth]{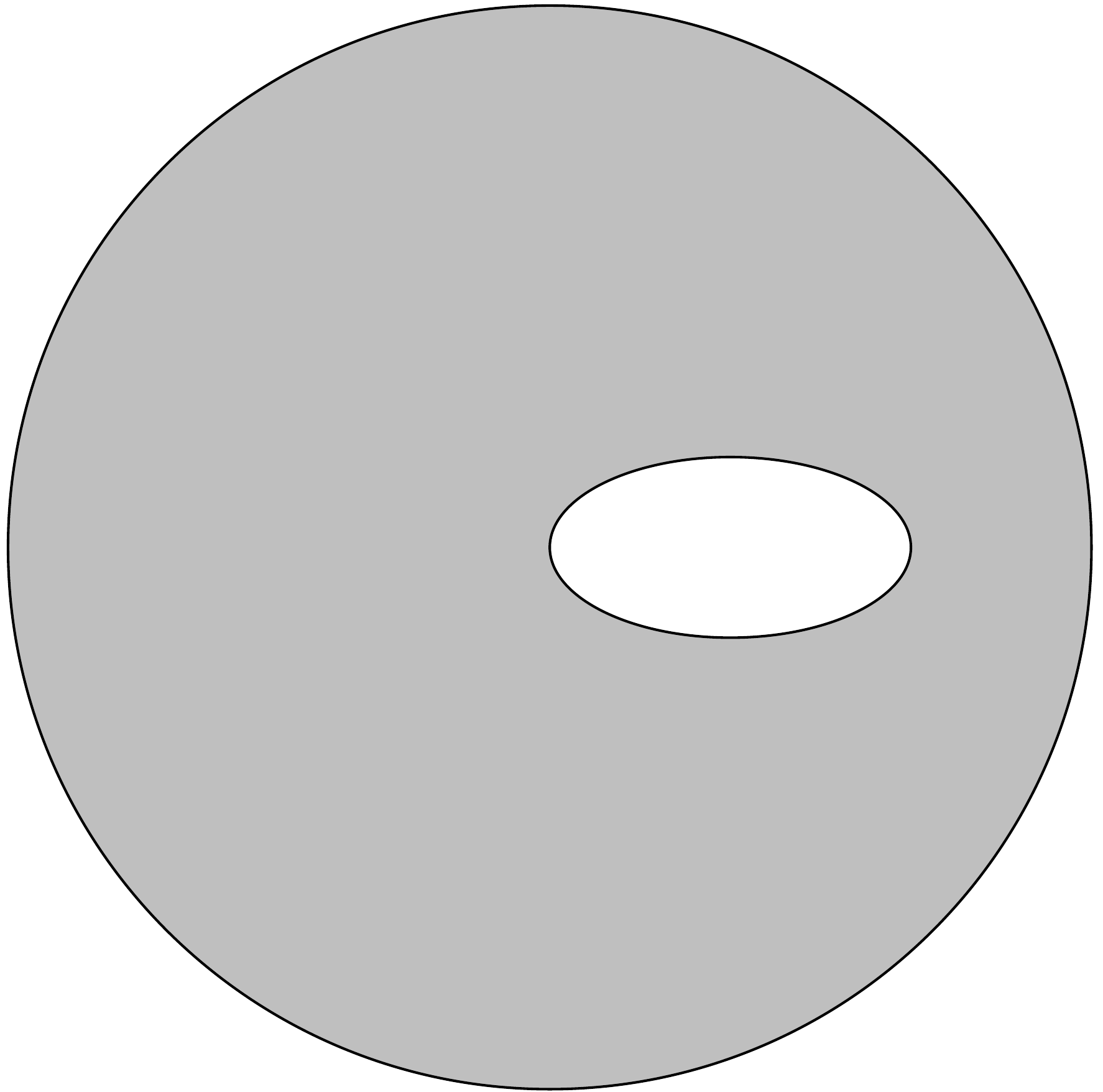}}
\caption{OFF and ON configurations for obstacles having $\mathbb{D}_4$ symmetry}\label{fig:on_off}
\end{figure}

\begin{figure}[H]\centering
\subfloat[OFF 
]{
\label{Off position_triangle}
\includegraphics[width=0.1\textwidth]{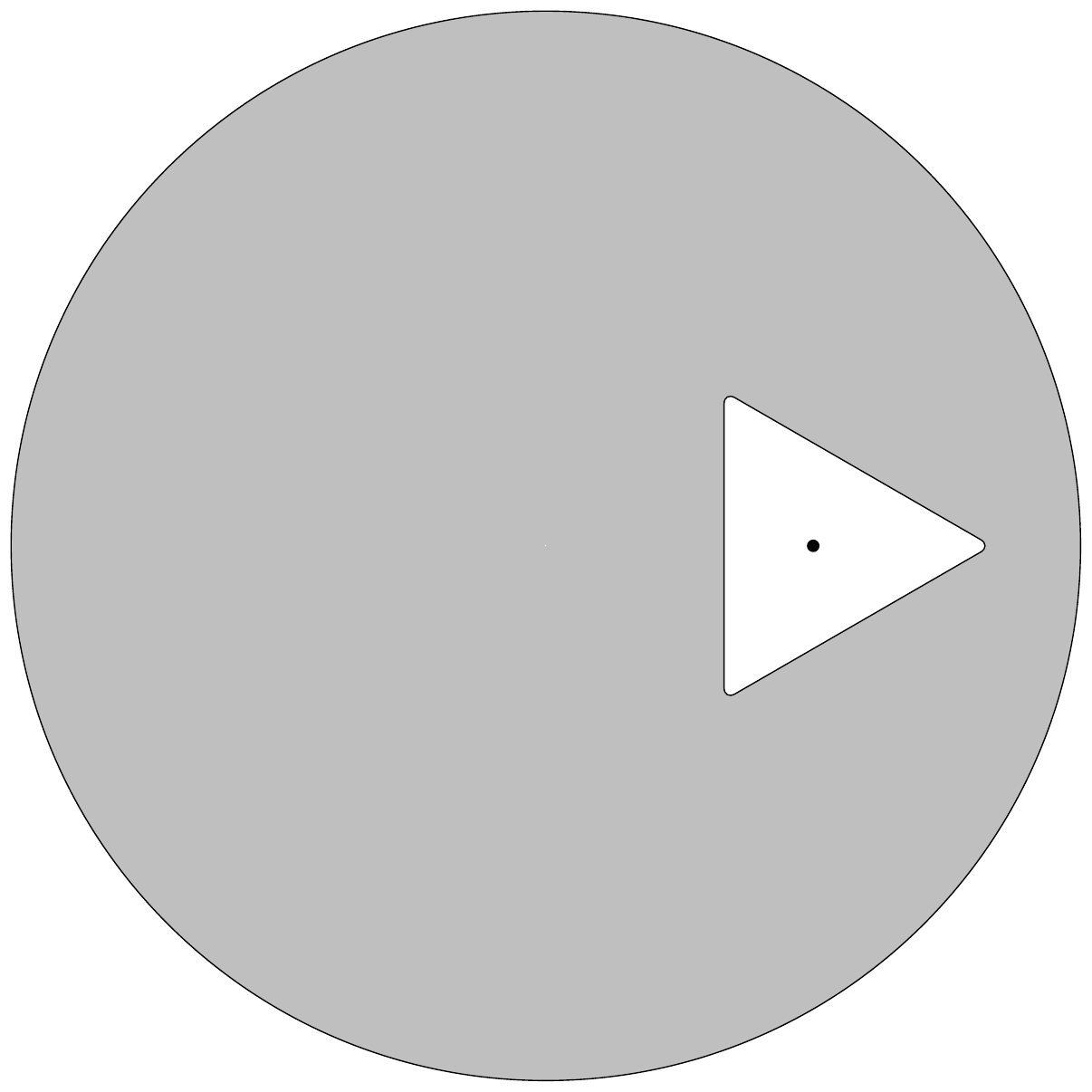}}
\hspace{12mm}
\subfloat[ON 
]{
\label{On position_triangle}
\includegraphics[width=0.1\textwidth]{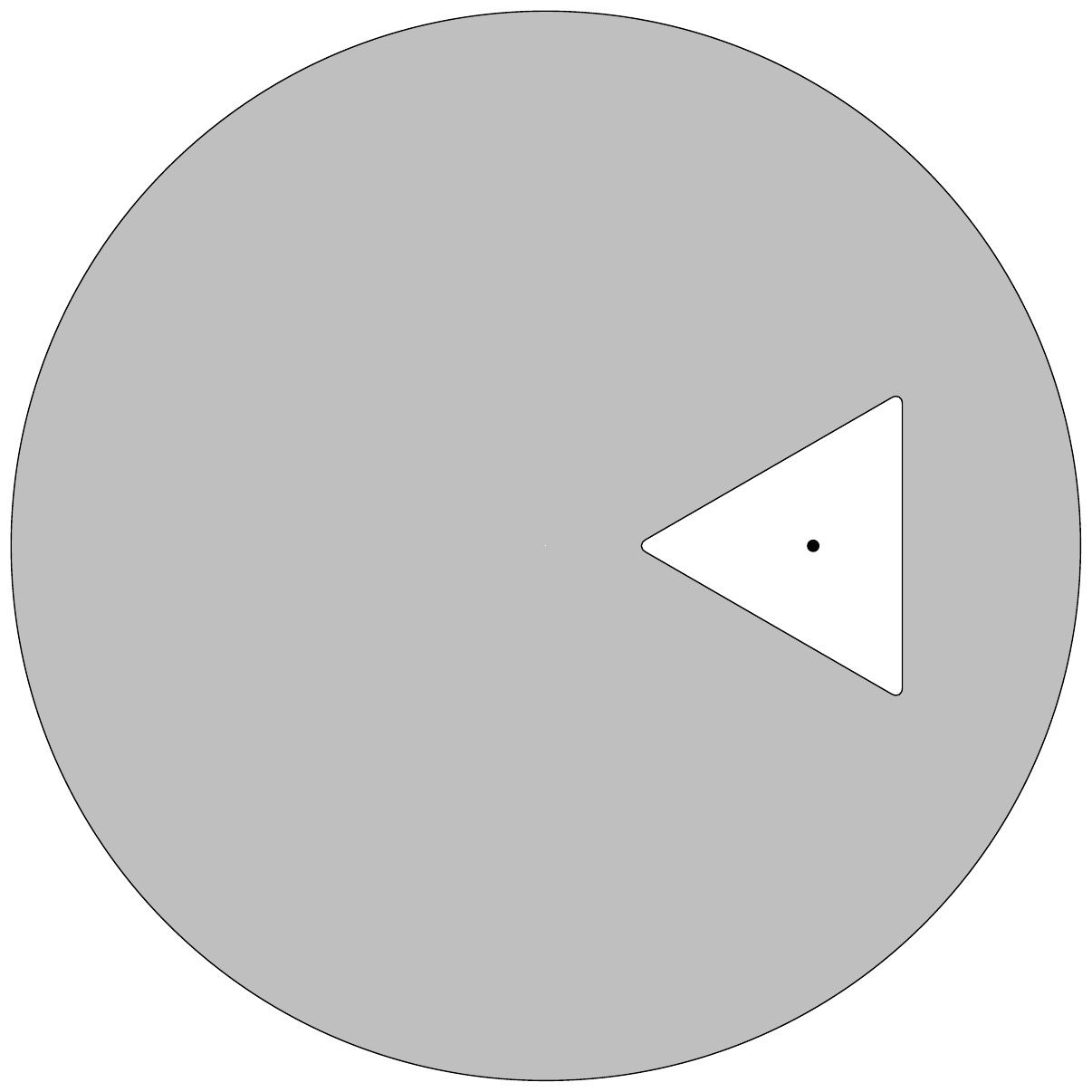}}
\hspace{14mm}
\subfloat[OFF 
]{
\label{Off position_pentagon}
\includegraphics[width=0.1\textwidth]{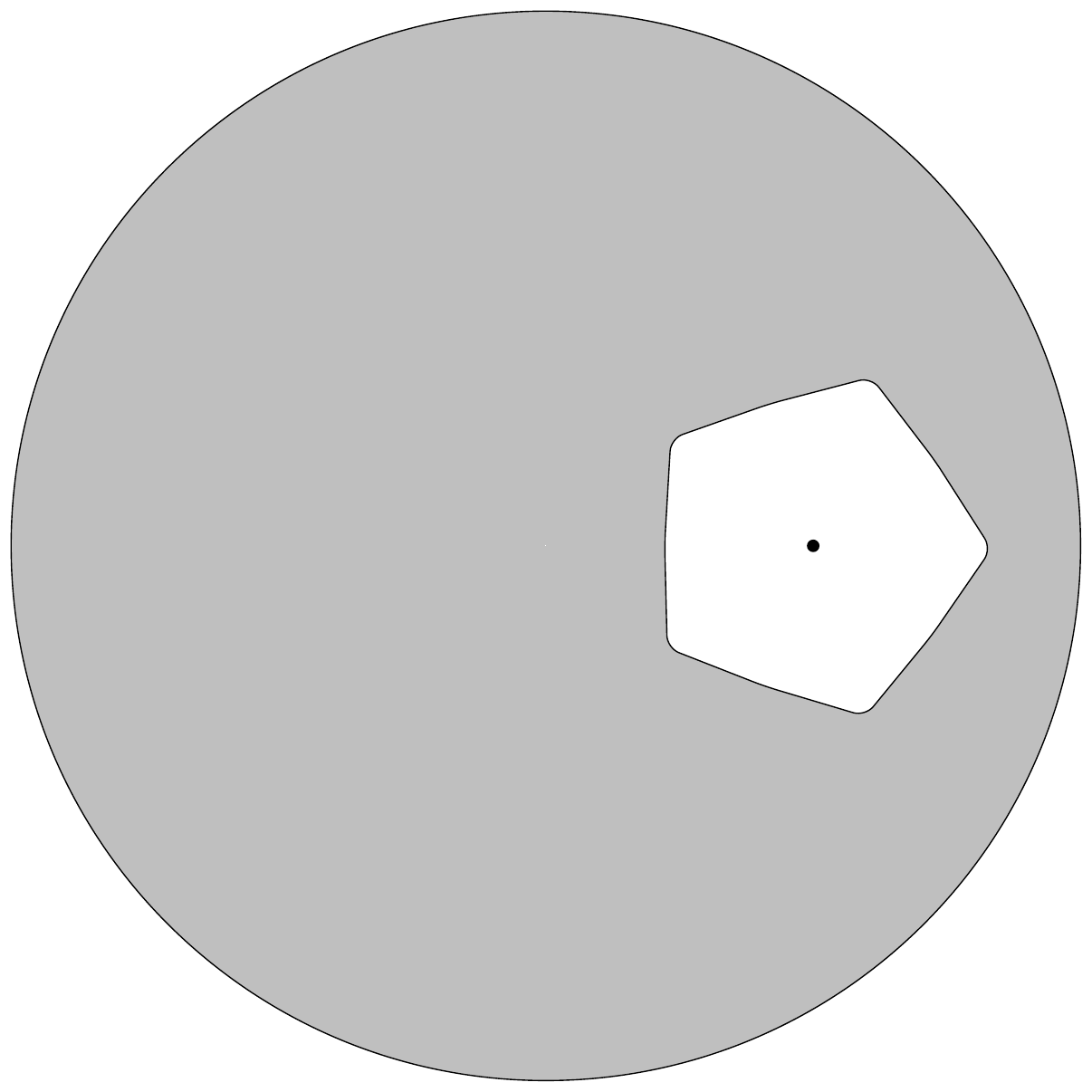}}
\hspace{12mm}
\subfloat[ON 
]{
\label{On position_pentagon}
\includegraphics[width=0.1\textwidth]{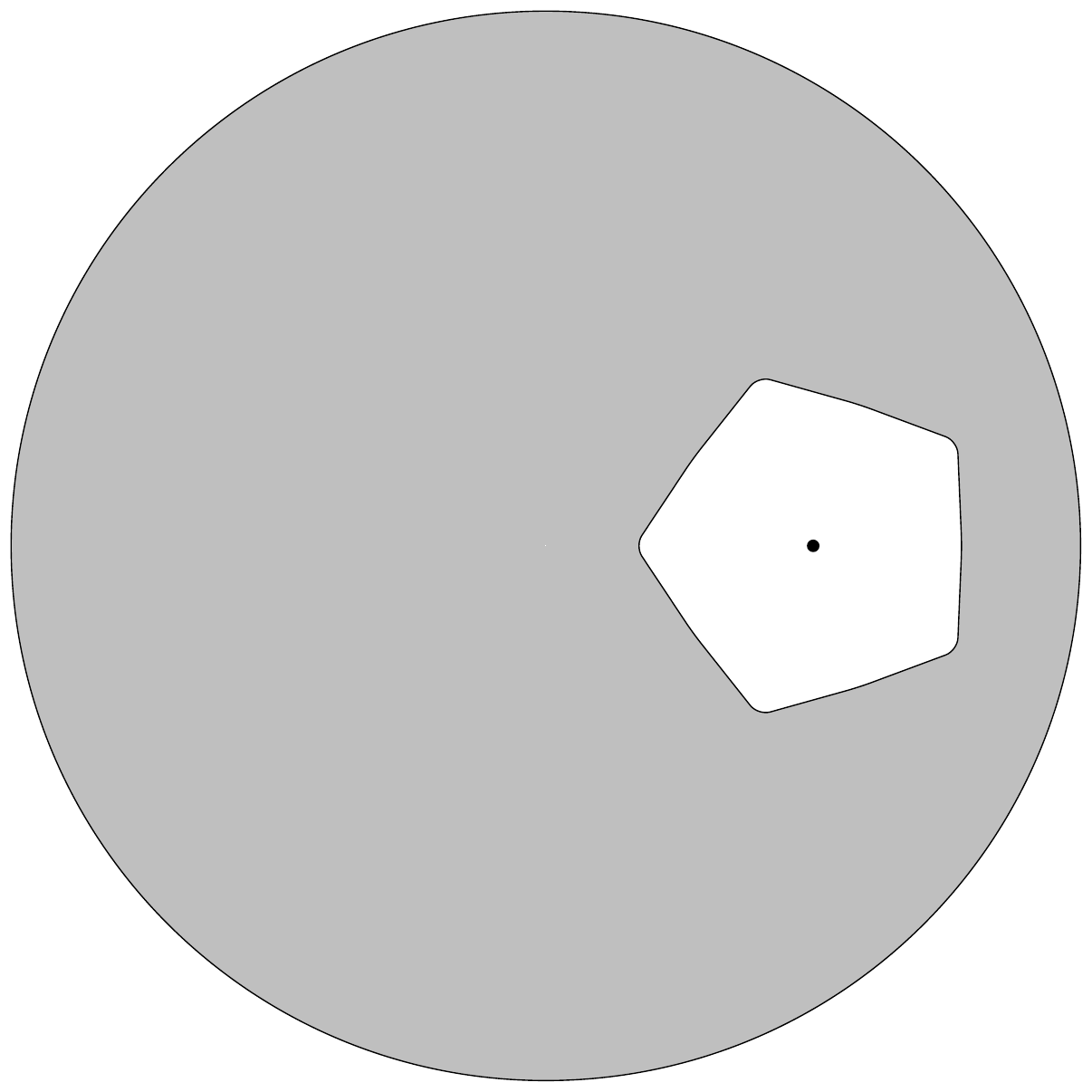}}
\caption{OFF and ON configurations for obstacles having $\mathbb{D}_n$ symmetry, $n$ odd}\label{fig:on_off_odd}
\end{figure}
{
\section{The Main Theorems} \label{stmnt}
We recall here that $P$ is a compact and simply connected subset of $\mathbb{R}^2$ satisfying assumptions (0) to (iv) given in section \ref{assumptions}. 
and that $B$ is an open disk in $\mathbb{R}^2$ of radius $r_1$ such that $B \supset cl(conv(C_2(P)))$. For $t\in\mathbb{R}$, let $\rho_t \in SO(2)$ denote the rotation in $\mathbb{R}^2$ about the origin $\underline{o}$ in the anticlockwise direction by an angle $t$, i.e., for $\zeta \in \mathbb{C} \cong \mathbb{R}^2  $, we have $\rho_t \zeta :=e^{\textbf{i} t} \zeta$. Now, fix $t \in [0, 2\pi[$. 
Let $\Omega_t:= B \setminus \rho_t(P)$ and $\mathcal{F}:= \{\Omega_t \, :\, t \in [0,2 \pi)\}$. 

\begin{theorem}[Extremal configurations w.r.t. the rotations of the obstacle about its fixed center away from the center of the disk]\label{max_min}
Fix $n \in \mathbb{N}, n$ even, $n \geq 3$. The energy functional $E(\Omega_t)$ for $\Omega_t \in \mathcal{F}$ is optimal precisely for those $t \in [0,2\pi[$ for which an axis of symmetry of $P_t$ coincides with a diameter of $B$. 

Among these optimal configurations, the maximizing configurations are the ones corresponding to those $t \in [0,2\pi[$ for which $P_t$ is in an OFF position with respect to $B$; and the minimizing configurations are the ones corresponding to those $t \in [0,2\pi[$ for which $P_t$ is in an ON position with respect to $B$.
\end{theorem}
{Equation (\ref{even_lambda}), Propositions \ref{critical_points} and \ref{complete_critical_points} imply Theorem \ref{max_min} for $n$ even, $n \in \mathbb{N}$, $n \geq 3$. For the $n$ odd case, we identify some of the extremal configuration for $E$. We prove that equation (\ref{even_lambda}) and Proposition \ref{critical_points} hold true for $n$ odd too. 
We provide numerical evidence for $n=5$ and conjecture that Proposition \ref{complete_critical_points}, and hence, Theorem \ref{max_min} hold true for $n$ odd too.}

Let $r_0^1$ and $r_0^2$ denote the radii of the incircle $C_1$ and the circumcircle $C_2$ of the obstacle $P$ respectively. Let $P_{(d,t)}$ be the obstacle $P_t$ as in Theorem \ref{max_min} with its center $\underline{o}$ at a distance $d < r_1-r_0^2$ from the center of $B$. Please note that, in Theorem \ref{max_min}, $d$ is a fixed number in $(0, r_1-r_0^2)$. 
This was because, for the case $d=0$, $t \longmapsto E(\Omega_t)$ is a constant map. This can be seen as follows. When $d=0$ we have  (a) $\Omega_t $ is isometric to $\Omega_s$ for each $t,s \in \mathbb{R}$, and (b) since the boundary data on each of the boundary component is radial w.r.t. the center of $B$ we get the solution $y$ of (\ref{laplace_equation_lc}) satisfies $y(\Omega_t) =y(\Omega_s)$  for each $t,s \in \mathbb{R}$. 
 
Since we want to study the behaviour of $E$ w.r.t. the translations of the obstacle too, we now allow $d$ to be $0$. Let $\Omega_{(d,t)}:= B \setminus P_{(d,t)}$ for $d \in [0, r_1-r_o^2)$, $t \in [0, 2\pi[$. Let $E((d,t)):=E(\Omega_{(d,t)})$. Let $\mathcal{G}$ be defined as $\{\Omega_{(d,t)} \, :\, (d,t) \in [0,r_1-r_0^2[ \times [0,2\pi[\}$. 
%
%
%
\begin{theorem}[Global extremal configurations, i.e., extremal configurations w.r.t. the translations and rotations of the obstacle within $B$]  \label{global}
 Fix $n \geq 3$, $n \in \mathbb{N}$. The concentric configuration, i.e., $\Omega_{(0, t)}$, for any $t \in [0,2\pi[$, is the minimising configuration for $E((d,t))$ over $\mathcal{G}$. At a maximising configuration for $E(d,t)$ over $\mathcal{G}$, the circumcircle of the obstacle must intersect or touch $\partial B$. 

For $n$ even, $n \geq 3$, $n \in \mathbb{N}$, we further have that, at the maximizing configuration over $\mathcal{G}$, the obstacle must be in an OFF position w.r.t. $B$. 

 The concentric configurations are the global minimising configurations w.r.t all the translations and all rotations of the obstacle within $B$, for each natural number $n \geq 3$. 
 The OFF configurations with an outer vertex touching $\partial B$ 
 are the global maximising configurations w.r.t. all the translations and rotations of the obstacle within $B$, for each even natural number $n \geq 3$. 
\end{theorem}
Let us study the behaviour of $E$ w.r.t. the expansion or contractions of the obstacle too. For this we fix $d \geq 0$ and $t \in [0, 2 \pi)$ such that $P_{(d,t)}$ lies completely inside $B$. We now scale $P=P_{(d,t)}$ by a scaling factor $\lambda >0$, $ \lambda \in \mathbb{R}$, 
such that $\lambda \, P_{(d, t)} := \{\lambda \, x \,|\, x \in P_{(d, t)} \}$ still lies completely inside $B$. Let $\Omega_{\lambda}:= B \setminus \lambda \, P_{(d,t)}$. 
Let $E(\lambda):=E(\Omega_{\lambda})$. Let $\mathcal{H}$ be defined as $\{\Omega_{\lambda} \, :\, \lambda>0, \lambda \in \mathbb{R}, 
\lambda \, P \subset B\}$. We denote the largest admissible value of $\lambda$ by $\lambda_P$ which depends on $d$ and $t$ both.
\begin{theorem}[Behaviour of the energy functional w.r.t. the scaling of the obstacle inside $B$]  \label{expansion}
 Fix $n \in \mathbb{N}, n\geq 3$. The energy functional strictly increases as the obstacle $P$ expands inside the disk $B$ and strictly decreases as the obstacle $P$ contracts inside $B$. Further, the energy functional approaches its minimum value as the obstacle $P$ shrinks to a point. 


\end{theorem}
\begin{remark} 
\begin{itemize} 
\item[(a)] It's worth mentioning here that the BVP  (\ref{laplace_equation_lc}) does not admit any solution when $P$ shrinks to the center of $B$. 
Also, a harmonic function on any planar disk (without any hole or puncture) with constant boundary condition always has zero energy. 
\item[(b)] Please note that, from the expression of the energy functional in (\ref{M}) and making use of the Hopf lemma, it is easy to see that the energy functional associated to (\ref{laplace_equation_lc}) increases as $|M|$, the absolute value of the boundary data, increases. In fact, for each $\alpha \in \mathbb{R}$, the energy functional scales by a factor of $\alpha^2$ as the value of the boundary data $g$ scales from $M$ to $\alpha \, M$. This follows from the fact that if $u$ is a solution of (\ref{laplace_equation_lc}) corresponding to the boundary data $g=M$, then $\alpha\, u$ becomes the solution of (\ref{laplace_equation_lc}) corresponding to the boundary data $g=\alpha \, M$, $\alpha \in \mathbb{R}$. 
\end{itemize}
\end{remark}
\section{Shape calculus}\label{scal}
\subsection{Existence of shape derivatives}\label{ExistSD} 
Let $D$ be a given domain in $\mathbb{R}^N$. Let $\Omega$ be a domain of class $\mathcal{C}^k$ in $D$. Let $V \in \mathcal{C}(0, \epsilon; \mathcal{D}^k(D; \mathbb{R}^N))$ be a vector field. 
Consider the following Dirichlet boundary value problem:
\begin{equation}\label{BVP119}
\begin{aligned}
-\Delta y(\Omega) &= h(\Omega) &\mbox{ in } & L^2(\Omega),\\
y(\Omega) &= z(\Gamma) &\mbox{ on } & \Gamma. 
\end{aligned}
\end{equation}
Proposition 3.1 on page 119 of \cite{Sokolowski_Zolesio} says that
\begin{proposition}\label{page119} Let $(h(\Omega), z(\Gamma)) \in L^2(\Omega) \times H^{\frac{1}{2}}(\Gamma)$ be given elements such that there exists the shape derivatives $(h'(\Omega)$, $z'(\Gamma))$ in  $L^2(\Omega) \times H^{\frac{1}{2}}(\Gamma)$. Then the solution $y(\Omega)$ to the Dirichlet boundary value problem (\ref{BVP119}) has the shape derivative $y^\prime(\Omega, V)$ in $H^1(\Omega)$ determined as the unique solution to the Dirichlet boundary value problem (\ref{1192}).
\begin{equation}\label{1192}
\begin{aligned}
-\Delta y^\prime(\Omega, V) &=h^\prime(\Omega, V) & \mbox{ in } & \mathcal{D}^\prime(\Omega),\\
y^\prime(\Omega, V)\arrowvert_{\Gamma} &=- \dfrac{\partial y}{\partial n} \left<V(0),n\right> + z^\prime(\Gamma, V)& \mbox{ on } & \partial \Omega,\\
\end{aligned}
\end{equation}
\end{proposition}
Please note that, for us the vector field $V$ is a function of $x$ and does not depend on $t$. Therefore, $V(0) =V$.

We observe that, for 
the Boundary value Problem (\ref{laplace_equation_lc}) on $\Omega =B \setminus P$, the corresponding $(h,z)$ belongs to $L^2(\Omega) \times H^{\frac{1}{2}}(\Gamma)$. And hence 
there exists a solution $y(\Omega) \in H^1(\Omega)$ of the Dirichlet boundary value problem (\ref{laplace_equation_lc}). We will prove that, for this choice of $h ,z$ and for the solution $y$, the shape derivatives $h^\prime(\Omega, V), z^\prime(\Gamma, V)$ and $y^\prime(\Omega, V)$ exist, and belong to $ L^2(\Omega), H^{\frac{1}{2}}(\Gamma)$ and $ H^1(\Omega)$ respectively, for any $V \in \mathcal{C}(0, \epsilon; \mathcal{D}^k(D; \mathbb{R}^N))$. We will also prove that both $h^\prime$ and $z^\prime$ are zero. Then, by the proposition \ref{page119} mentioned above, i.e., Proposition 3.1 on page 119 of \cite{Sokolowski_Zolesio} we would have proved that the shape derivative $y^\prime(\Omega, V)$ satisfies boundary value problem (\ref{1192new}). 


In view of Definition 2.71 on page 98 of \cite{Sokolowski_Zolesio}, the material derivative $\dot{h}=\dot{h}(\Omega, V) 
$ of the function $h(\Omega) \in C^\infty(\Omega)$ exists, belongs to $C^\infty(\Omega)$ and equals zero. Moreover, because $h \in  C^\infty(\Omega)$, $\dot{h}(\Omega, V)$ is exactly equal to $\left<\nabla h, V\right>$. Therefore, by Definition 2.85 on page 111 of \cite{Sokolowski_Zolesio}, we get the shape derivative $h^\prime$ of $h$ in the direction of $V$ exists and is an element of $ \mathcal{C}^\infty(\Omega)$ defined by $h^\prime(\Omega, V)= \dot{h}(\Omega, V) - \left<\nabla h, V\right>$ is $0$.

In view of Definition 2.74 on page 100 of \cite{Sokolowski_Zolesio}, the material derivative $\dot{z}=\dot{z}(\Gamma, V) 
$ of $z(\Gamma) \in C^\infty(\Gamma)$ exists, belongs to $C^\infty(\Gamma)$ and equals zero. Moreover, because $z \in  C^\infty(\Gamma)$, $\dot{z}(\Gamma, V)$ is exactly equal to $\left<\nabla_\Gamma z, V\right>$. Therefore, by Definition 2.88 on page 114 of \cite{Sokolowski_Zolesio}, we get the shape derivative of $z^\prime$ in the direction of $V$ exists and is an element of $ \mathcal{C}^\infty(\Gamma)$ defined by $z^\prime(\Gamma, V)= \dot{z}(\Gamma, V) - \left<\nabla_\Gamma z(\Gamma), V\right>$ is $0$.

Thus, (\ref{1192}) becomes
\begin{equation}\label{1192new}
\begin{aligned}
-\Delta y^\prime(\Omega, V) &=0 & \mbox{ in } & \mathcal{D}^\prime(\Omega),\\
y^\prime(\Omega, V)\arrowvert_{\Gamma} &=- \dfrac{\partial y}{\partial n} \left<V,n\right>& \mbox{ on } & \partial \Omega,\\
\end{aligned}
\end{equation}

\subsection{The energy functional} 
We recall here that $P$ is a compact and simply connected subset of $\mathbb{R}^2$ satisfying assumptions (0) to (iv) of section \ref{assumptions}. 
and that $B$ is an open disk in $\mathbb{R}^2$ of radius $r_1$ such that $B \supset cl(conv(C_2(P)))$. For $s\in\mathbb{R}$, let $\rho_s \in SO(2)$ denote the rotation in $\mathbb{R}^2$ about the origin $\underline{o}$ in the anticlockwise direction by an angle $s$, i.e., for $\zeta \in \mathbb{C} \cong \mathbb{R}^2  $, we have $\rho_s \zeta :=e^{\textbf{i} s} \zeta$. Now, fix $s \in [0, 2\pi[$. 
Let $P_s :=  \rho_s(P)$, $\Omega_s:= B \setminus \rho_s(P)$ and $\mathcal{F}:= \{\Omega_s \, :\, s \in [0,2 \pi)\}$. 

When $g=M$, a constant, then 
by Theorem 8.14 on page 188 of \cite{Gilberg-Trudinger} 
we get that the solution $u_s$ of (\ref{laplace_equation_lc}) on $\Omega_s$ belong to $\mathcal{C}^\infty(\overline{\Omega_s})$. Therefore, by Green's identity we get 
$E(s):= E(\Omega_s) = \int_{\Omega_s} \|\nabla u_s\|^2 \, dx = - \int_{\Omega_s}( \Delta u_s) u_s \, dx + \int_{\partial \Omega_s}  u_s \, \dfrac{\partial u_s}{\partial n} \, d\Sigma,$
where $d\Sigma$ is the line element on $\partial \Omega_s$ and $n$ is the outward unit normal vector to $\Omega_s$ at $x \in \partial\Omega_s$. 
From (\ref{laplace_equation_lc}), $E(s)$ reduces to the following
\begin{equation}\label{M}E(s) = \int_{\partial B} u_s \, \dfrac{\partial u_s}{\partial n} \, d\Sigma = \int_{\partial B} g \, \dfrac{\partial u_s}{\partial n} \, d\Sigma =M\,\int_{\partial B} \dfrac{\partial u_s}{\partial n} \, d\Sigma .\end{equation}
\subsection{Formal deduction of the Eulerian derivative of the energy functional}\label{deduction}
By the arguments similar to the one given in section \ref{ExistSD}, one can prove that, for each $s$ such that $s \in [0, 2\pi[$, the shape derivative of $u_s$ w.r.t. the perturbation vector field $V$ exists 
 and satisfies the following boundary value problem:
\begin{equation}\label{shapederivativeBVP}
\begin{aligned}
\Delta w &=0 & \mbox{ in } & \Omega_s:=B \setminus \overline{P_s},\\
&w =- \dfrac{\partial u_s}{\partial n} \left<V,n\right> & \mbox{ on } & \partial P_s,\\
&w=0 & \mbox{ on } & \partial B.
\end{aligned}
\end{equation}
 We denote the shape derivative of $u_s$ as $u_s^\prime$.

 Since $u_s \in \mathcal{C}^\infty(\bar{\Omega})$ and $V \in \mathcal{C}^\infty(D)$, Theorem  8.14 on page 188 of \cite{Gilberg-Trudinger} 
implies that the weak solution $u^\prime_s \in H^1(\Omega)$ of (\ref{shapederivativeBVP}) is in fact in $\mathcal{C}^\infty(\bar{\Omega})$. 
Therefore, by Green's identity applied 
to $u_s$ and $u_s^\prime$ 
we get,
$\int_{\Omega_s} u_s \Delta u_s^\prime dx-\int_{\Omega_s} u_s^\prime \Delta u_s dx = \int_{\partial \Omega_s} u_s  \dfrac{\partial u_s^\prime}{\partial n} d\Sigma-\int_{\partial \Omega_s} u_s^\prime \dfrac{\partial u_s}{\partial n} d\Sigma .$
In view of (\ref{laplace_equation_lc}) and (\ref{shapederivativeBVP}), we then get
\begin{equation}\label{BoundInt}\int_{\partial B} u_s  \dfrac{\partial u_s^\prime}{\partial n} d\Sigma=\int_{\partial P_s} u_s^\prime \dfrac{\partial u_s}{\partial n} d\Sigma=-\int_{\partial P_s} \left( \dfrac{\partial u_s}{\partial n}\right)^2 \left<V, n\right>d\Sigma . \end{equation}

Since $E(s)= \int_{\partial \Omega_s} u_s \, \frac{\partial u_s}{\partial n}\, dS$, we refer to section 2.33 titled `Derivatives of boundary integrals' on page 115 of \cite{Sokolowski_Zolesio} to derive the expression of the Eulerian derivative of $E(\Omega)$: Let $D$ be a given domain in $\mathbb{R}^N$. Let $Z_s:=Z(\partial \Omega_s) := u_s \,\dfrac{\partial u_s}{\partial n}  \in L^1(\partial \Omega_s)$. From section \ref{ExistSD} it follows that for every vector field $V  \in \mathcal{C}(0, \epsilon; \mathcal{D}^k(D; \mathbb{R}^N))$, $u_s$ has a strong material derivative in $L^1(\partial \Omega_s)$ and a shape derivative in $L^1(\partial \Omega_s)$. Now since $u_s\in \mathcal{C}^\infty(\bar{\Omega})$, it is easy to see that $Z_s$ has a strong material derivative in $L^1(\partial \Omega_s)$ and a shape derivative in $L^1(\partial \Omega_s)$, for any vector field $V  \in \mathcal{C}(0, \epsilon; \mathcal{D}^k(D; \mathbb{R}^N))$. By equation (2.173) on page 116 of \cite{Sokolowski_Zolesio}, it follows that the Eulerian derivative $dE(\Omega, V)$ of $E$ at  
$\Omega$ in the direction $V$ exists and that
$dE(\Omega, V) = \int_{\partial \Omega} \left[ Z^\prime(\partial \Omega, V) + \kappa \, Z \left<V, n\right>\right]\, d\Sigma,$ where $\kappa$ is the mean curvature on the manifold $\partial \Omega$. It's not difficult to see that $Z^\prime(\partial \Omega, V)=u \,\dfrac{\partial u^\prime}{\partial n}+ u^\prime \,\dfrac{\partial u}{\partial n}$. Therefore, $dE(\Omega, V) = \int_{\partial \Omega} \left[ u \,\dfrac{\partial u^\prime}{\partial n}+ u^\prime \,\dfrac{\partial u}{\partial n} + \kappa \, u \, \dfrac{\partial u}{\partial n}\,\left<V, n\right>\right]\, d\Sigma.$
 Now, as $V\equiv 0$ on $\partial B$, from (\ref{laplace_equation_lc}) and (\ref{shapederivativeBVP}), it follows that 
$$dE(\Omega, V) = \int_{\partial B}  u \,\dfrac{\partial u^\prime}{\partial n} \, d\Sigma+  \int_{\partial P} 
u^\prime \,\dfrac{\partial u}{\partial n}
\left<V, n\right>
\, d\Sigma= \int_{\partial B}  u \,\dfrac{\partial u^\prime}{\partial n} \, d\Sigma-  \int_{\partial P} \left(\dfrac{\partial u}{\partial n}\right)^2 \,\left<V, n\right>\, d\Sigma.$$
Now, from (\ref{BoundInt}) it follows that 
$dE(\Omega, V) = -\int_{\partial P} \left( \dfrac{\partial u}{\partial n}\right)^2 \left<V, n\right>d\Sigma -  \int_{\partial P} \left(\dfrac{\partial u}{\partial n}\right)^2 \,\left<V, n\right>\, d\Sigma.$
Therefore,
\begin{equation}\label{EulerianDerivative}dE(\Omega, V) = -2\int_{\partial P} \left( \dfrac{\partial u}{\partial n}\right)^2 \left<V, n\right>d\Sigma.\end{equation}
In a similar way one can prove that the Eulerian derivative of $E$ at  
$\Omega_{s}$ in the direction $V$ exists for each $s \in [0, 2 \pi[$. We denote this Eulerian derivative by $dE(\Omega_s;V)$. Thus,  for each $s \in [0, 2 \pi[$,
\begin{equation}\label{EulerianDerivativeS1}dE(\Omega_s, V) = -2\int_{\partial P_s} \left( \dfrac{\partial u_s}{\partial n}\right)^2 \left<V, n\right>d\Sigma.\end{equation}
\section{Proof of Theorem \ref{max_min}}\label{proof}
In this section, we prove Theorem \ref{max_min} {for $n \in \mathbb{N}, n \geq 3$, $n$ even. We further prove that equation (\ref{even_lambda}) and Proposition \ref{critical_points} hold true for every natural number $n\geq 3$.}

We first justify that, {for any natural number $n \geq 3$}, the energy functional $E$ for the family of domains under consideration is a function of just one real variable, and that it is an even, differentiable and periodic function of period $2 \pi /n$. This helps in identifying the critical points of $E$.  Therefore, in order to determine the extremal configuration/s for $E$ we study its behaviour on the interval $[0, \frac{\pi}{n}]$. The expression (\ref{EulerianDerivativeS1}) for its derivative, that we derived in section \ref{deduction}, becomes useful in this analysis. We identify some of critical points of $E$ in Proposition \ref{critical_points} {for $n \in \mathbb{N}$, $n\geq 3$.}  

We prove Proposition \ref{complete_critical_points} {for even natural number $n$, $n \geq3$}. In view of equation (\ref{even_lambda}), Propositions \ref{critical_points} and \ref{complete_critical_points} imply that, {for $n \in \mathbb{N}$, $n$ even, $n \geq3$}, (a) these are the only critical points for $E$, and that, (b) between every pair of consecutive critical points, $E$ is a strictly monotonic function of the argument. {For this analysis, we use the `sector reflection technique' and the `rotating plane method' which were introduced in \cite{Anisa-Souvik}.}
\subsection{\textbf{Sufficient Condition for the Critical Points of $E(B \setminus P_t),$ $t \in [0, 2 \pi[$}} 
{ Fix $n \geq3$, $n \in \mathbb{N}$.} Let $E(t)$ be as in (\ref{energy_functional}) which denotes the energy functional associated to the Boundary value problem (\ref{laplace_equation_lc}) on $\Omega_t$, i.e., $E(t):= E(\Omega_t)$.  In this section, we establish a sufficient condition for the critical points of the $\mathcal{C}^1$ function $E: \mathbb{R} \rightarrow ]0, \infty[ $. 

Recall from \cite{Anisa-Souvik} 
that $B =\lbrace{re^{i\phi}: \phi\in[0,2\pi[, 0\leq r < g(\phi)\rbrace}$, where $g: [0,2 \pi] \rightarrow [0, \infty[$ is a $\mathcal{C}^2$ map with $g(0)= g(2 \pi)$. Here, $(r,\phi)$ is measured with respect to the origin $\underline{o}=(0,0)$ of $\mathbb{R}^2$. 
\subsubsection{\textbf{The Initial Configuration}}\label{a:init_config}
We start with the following initial configuration $\Omega_{\mbox{init}}$ of a domain $\Omega \in \mathcal{F}$. 
Let $P$ and $B$ be as described in section \ref{stmnt}. Let $\Omega_{\mbox{init}}$  denote the domain $B \setminus P \in \mathcal{F}$, where $P$ is in an OFF position with respect to $B$. { Recall that we assumed, without loss of generality, that (a) the centers of $B$ and $P$ are on the $x_1$-axis, (b) the center of $P$ is at the origin, and (c) the center of $B$ is on the negative $x_1$-axis. 
Let $\underline{x}^0:=(-x^0,0)$ be the center of the disk $B$, where $0< x^0<r_1$. }The initial configurations for obstacles with $\mathbb{D}_n$ symmetry are shown in Figure \ref{fig:in_config}. 
\begin{figure}[H]\centering
\subfloat{
\includegraphics[width=0.1\textwidth]{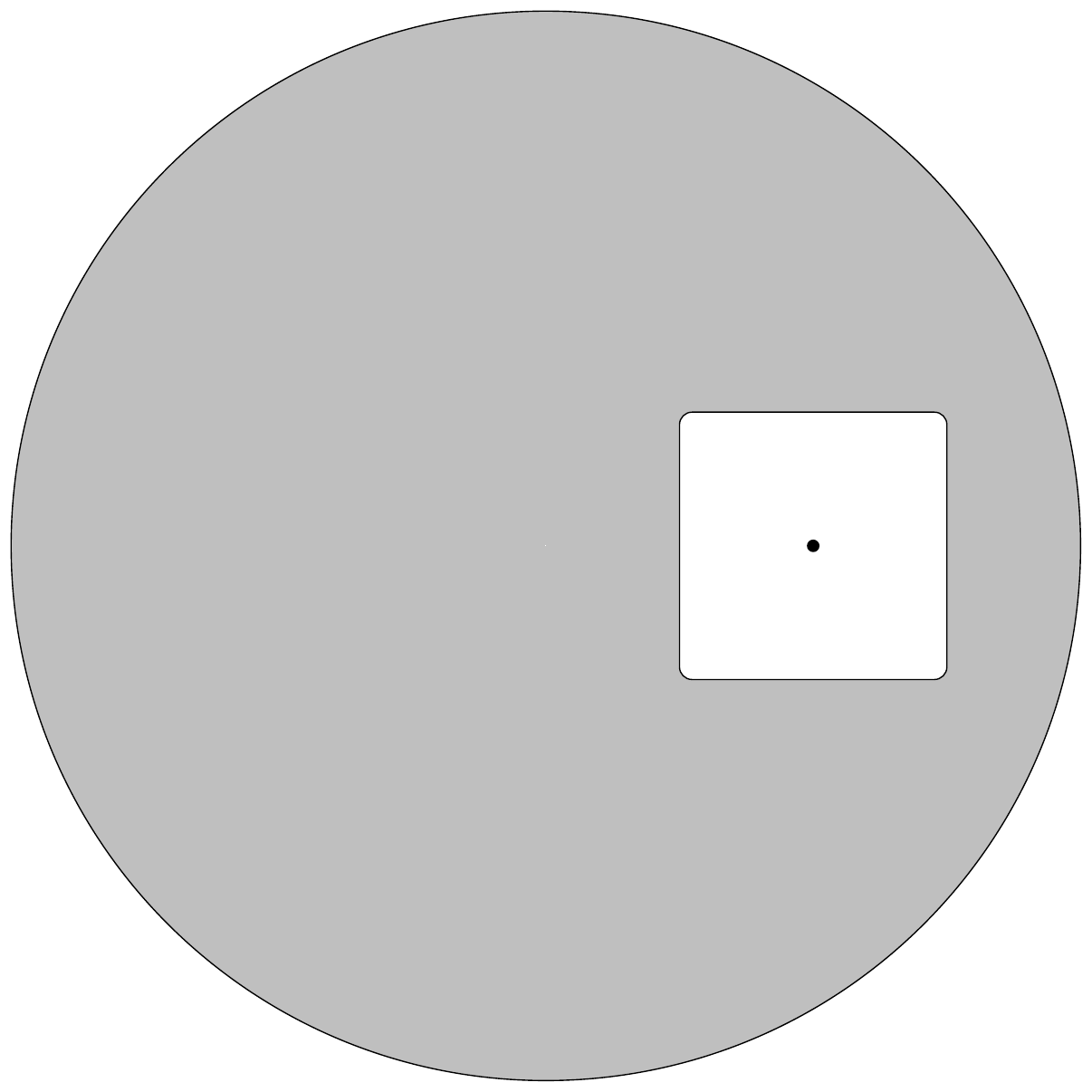}}\hspace{10mm}
\subfloat{
\includegraphics[width=0.1\textwidth]{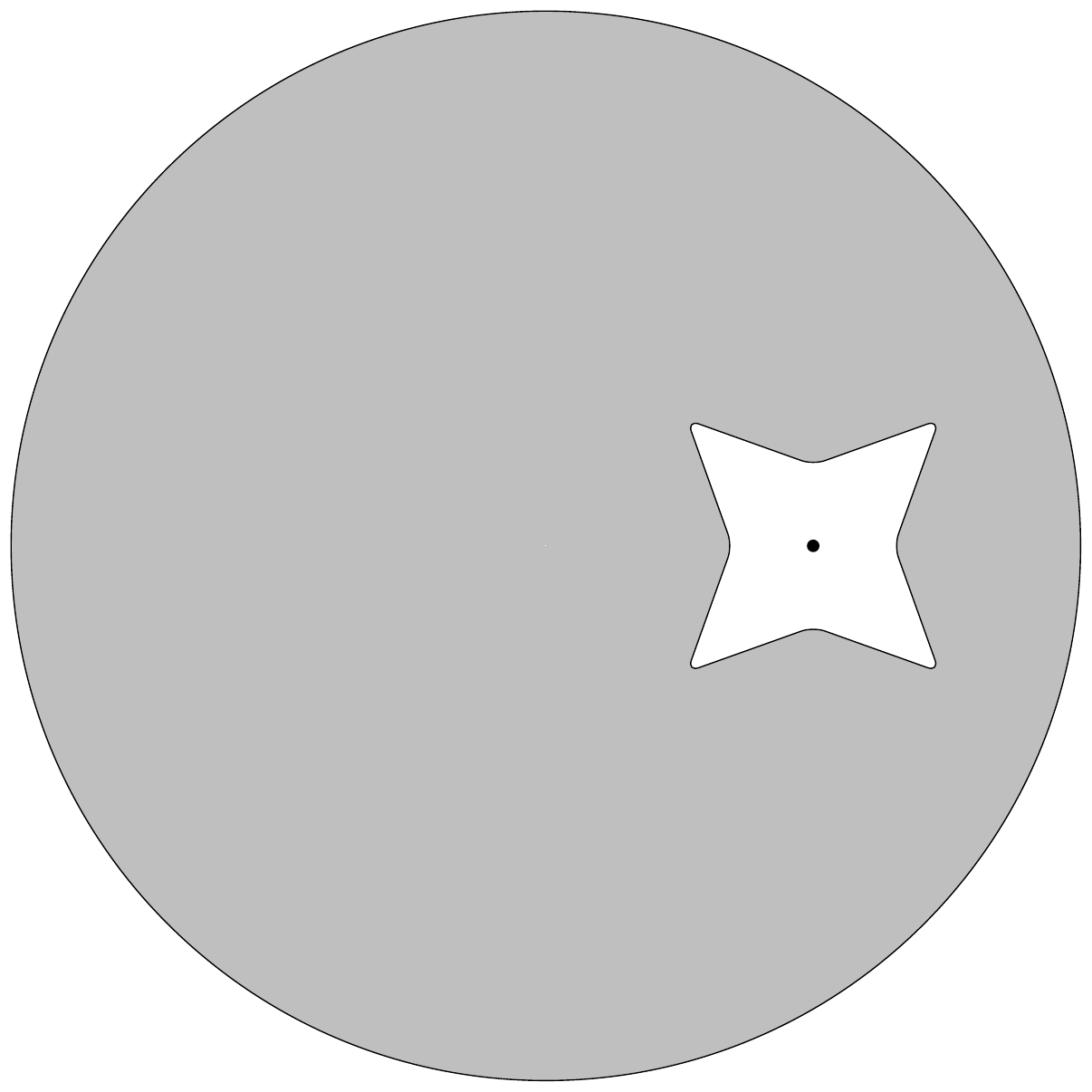}} 
\hspace{10mm}
\subfloat{
\includegraphics[width=0.1\textwidth]{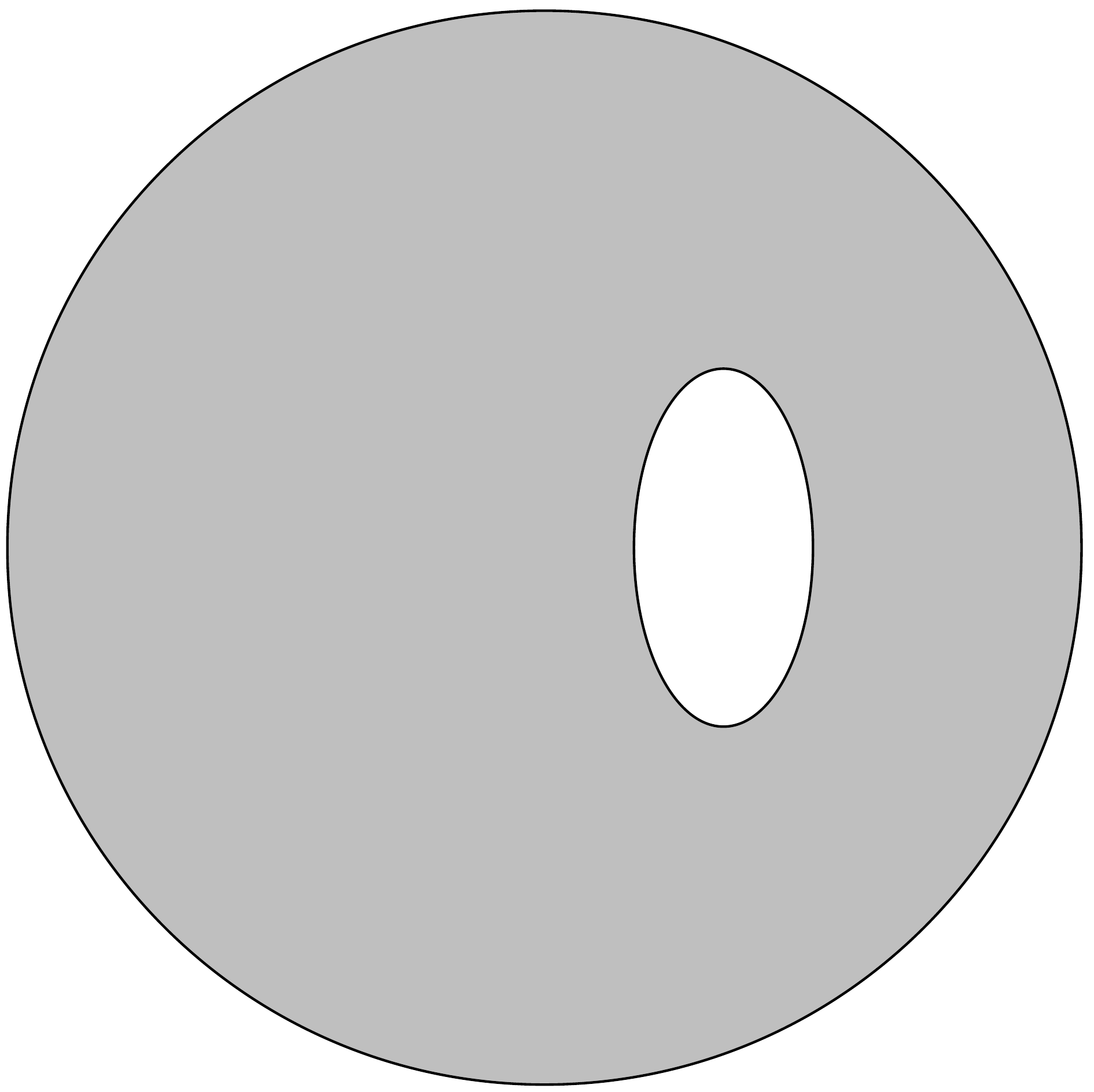}}
\hspace{10mm}
\subfloat{
\includegraphics[width=0.1\textwidth]{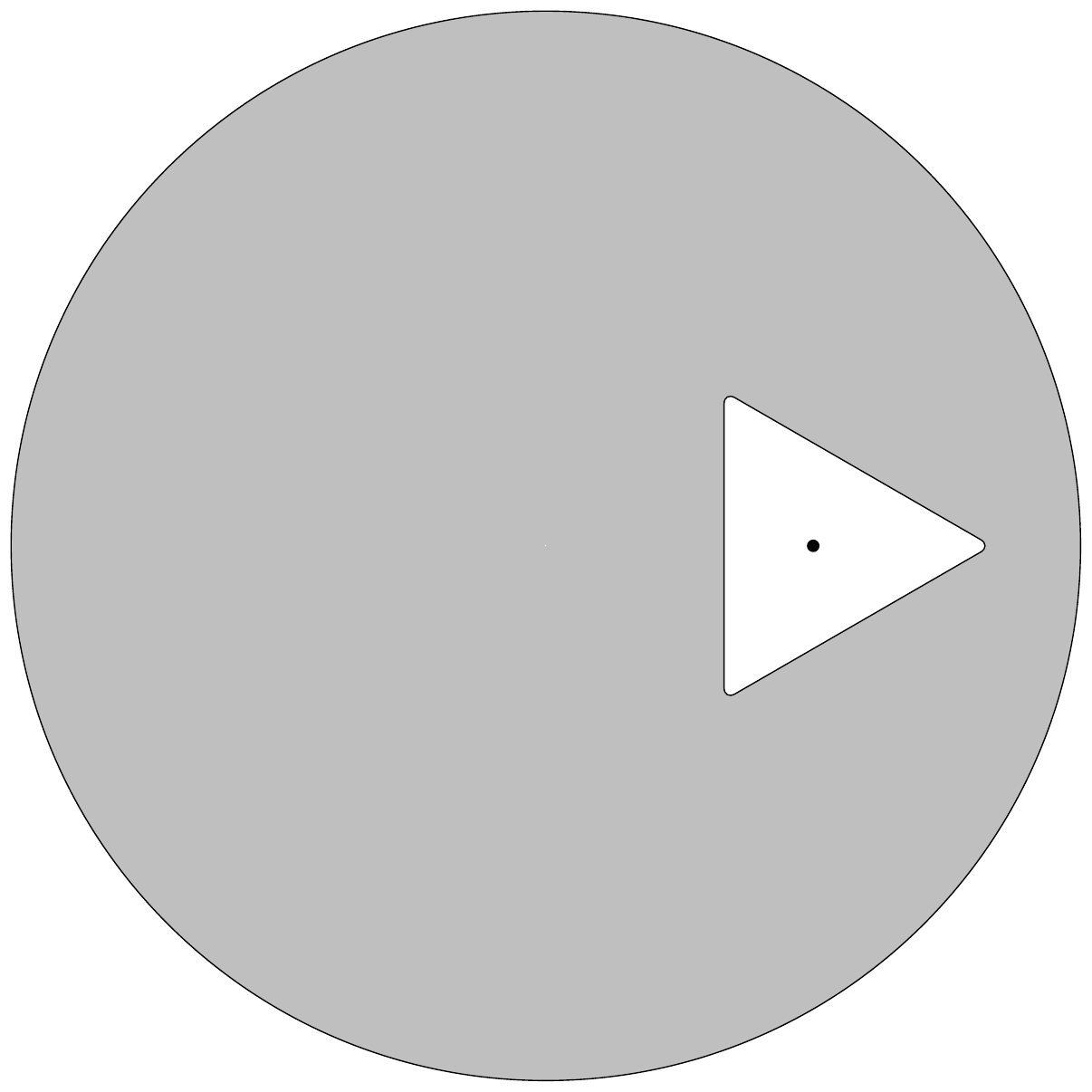}}\hspace{10mm}
\subfloat{
\includegraphics[width=0.1\textwidth]{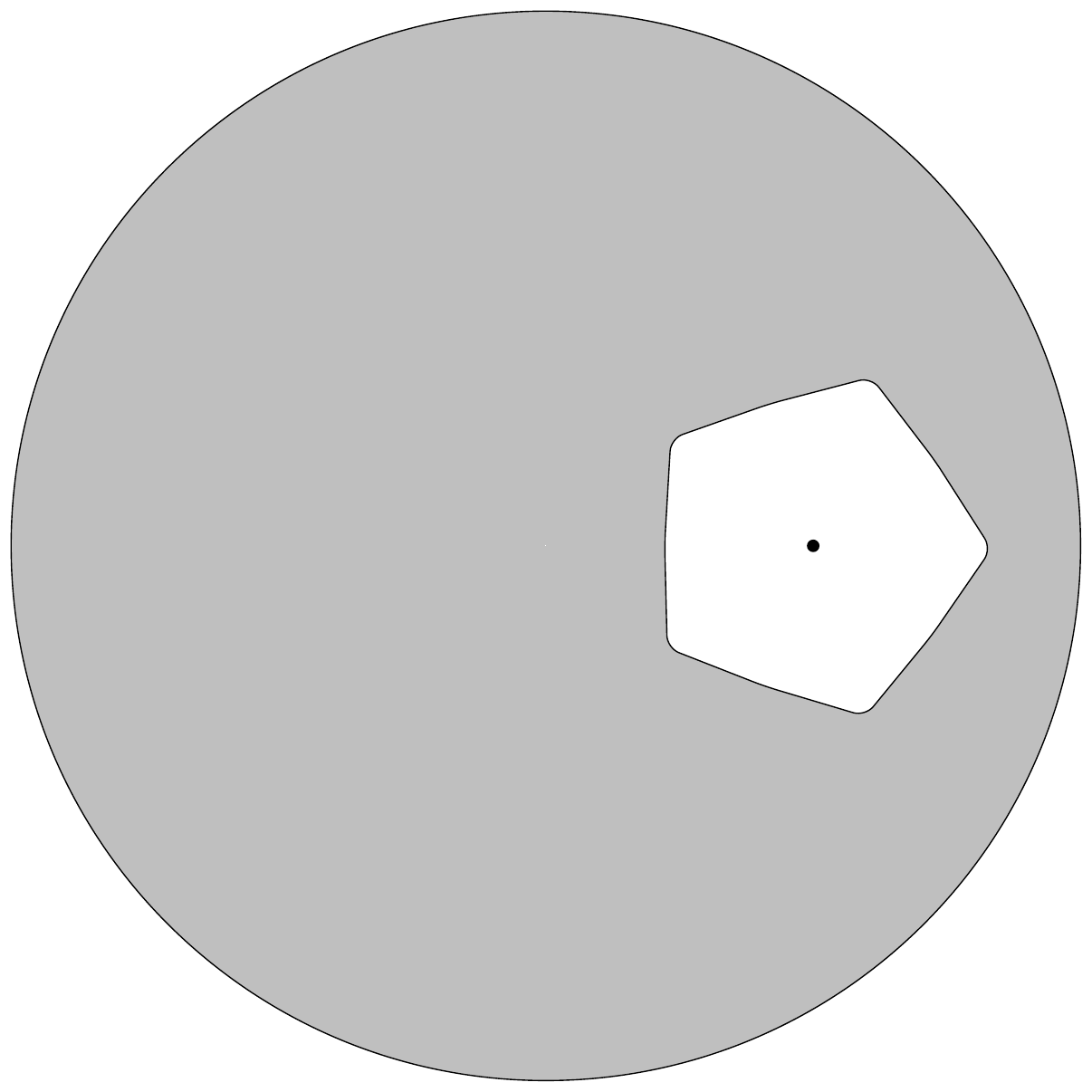}}
\caption{The initial configurations}
\label{fig:in_config}
\end{figure}
As in \cite{Anisa-Souvik}, 
$P=\lbrace re^{\textbf{i}\phi}\, :\, \phi\in [0,2\pi[, 0 \leq r < f(\phi)\rbrace$,
where $f: [0,2 \pi] \rightarrow [0, \infty)$ is a $\mathcal{C}^2$ map with $f(0)=f(2 \pi)$. Because of the initial configuration assumptions on $B \setminus P$, { $f$ 
is an increasing function of $\phi$ on $]{0,\frac{\pi}{n}}[$ for $n$ even, and is a decreasing function of $\phi$ on $]{0,\frac{\pi}{n}}[$ for $n$ odd. }
The condition that the obstacle $P$ can rotate freely about its center $\underline{o}$ inside $B$, that is, $\rho(P)$ is contained in $B~ \forall \rho \in SO(2)$ is guaranteed by assuming that the closure of the convex hull of the circumcircle $C_2(P)$ is contained in $B$.

\subsubsection{\textbf{Configuration at Time $t$}}
 Now, fix $t \in [0, 2\pi[$. We set 
$P_t := \rho_t(P), \Omega_t := B \setminus P_t$.
Then, in polar coordinates, we have $\partial P_t : =  \{f(\phi-t) e^{i \phi}\, |\, \phi \in [0, 2\pi[\}.$
\begin{figure}[H]\centering
\subfloat{\includegraphics[width=0.2\textwidth]{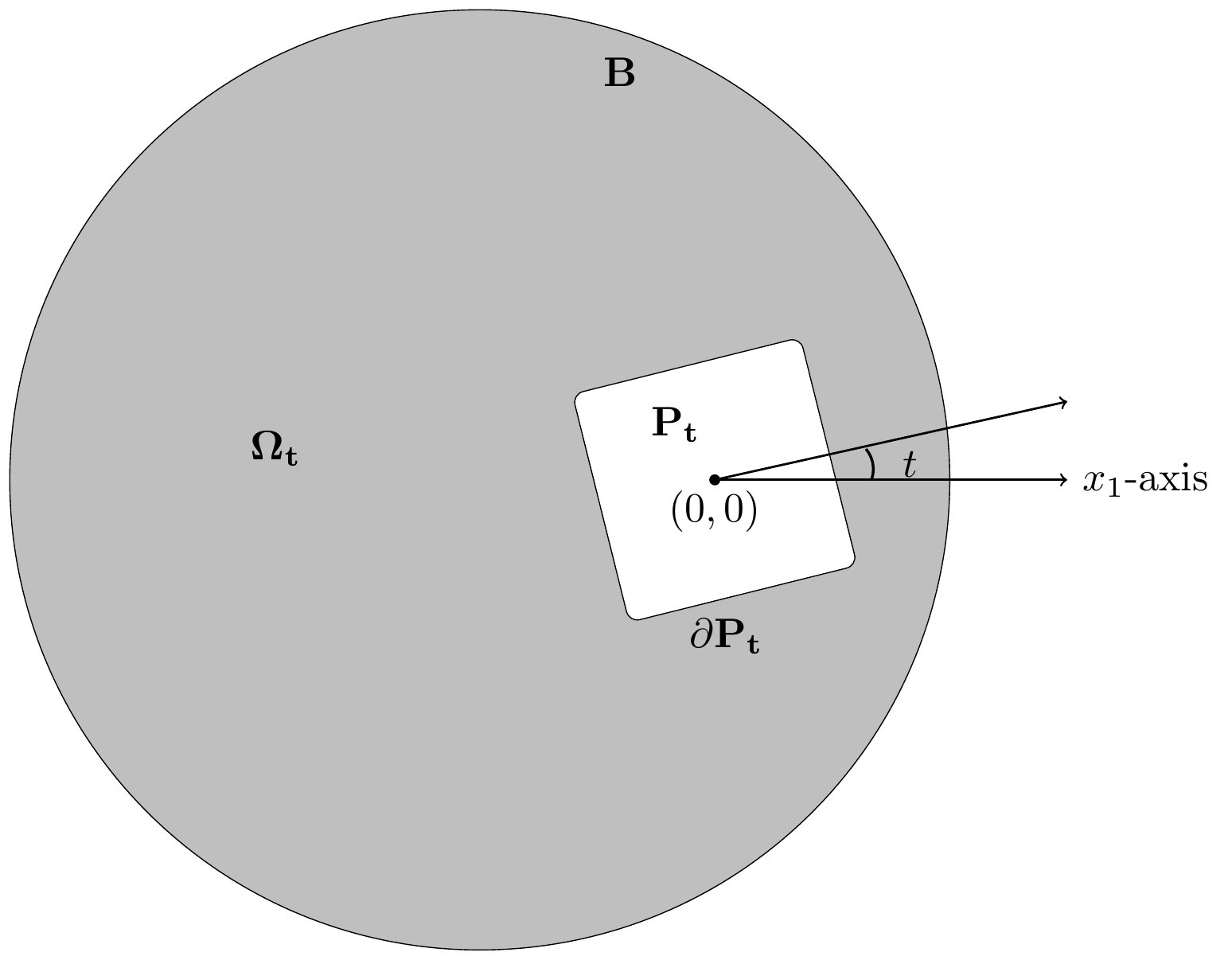}   }\hspace{8mm}
\subfloat{\includegraphics[width=0.2\textwidth]{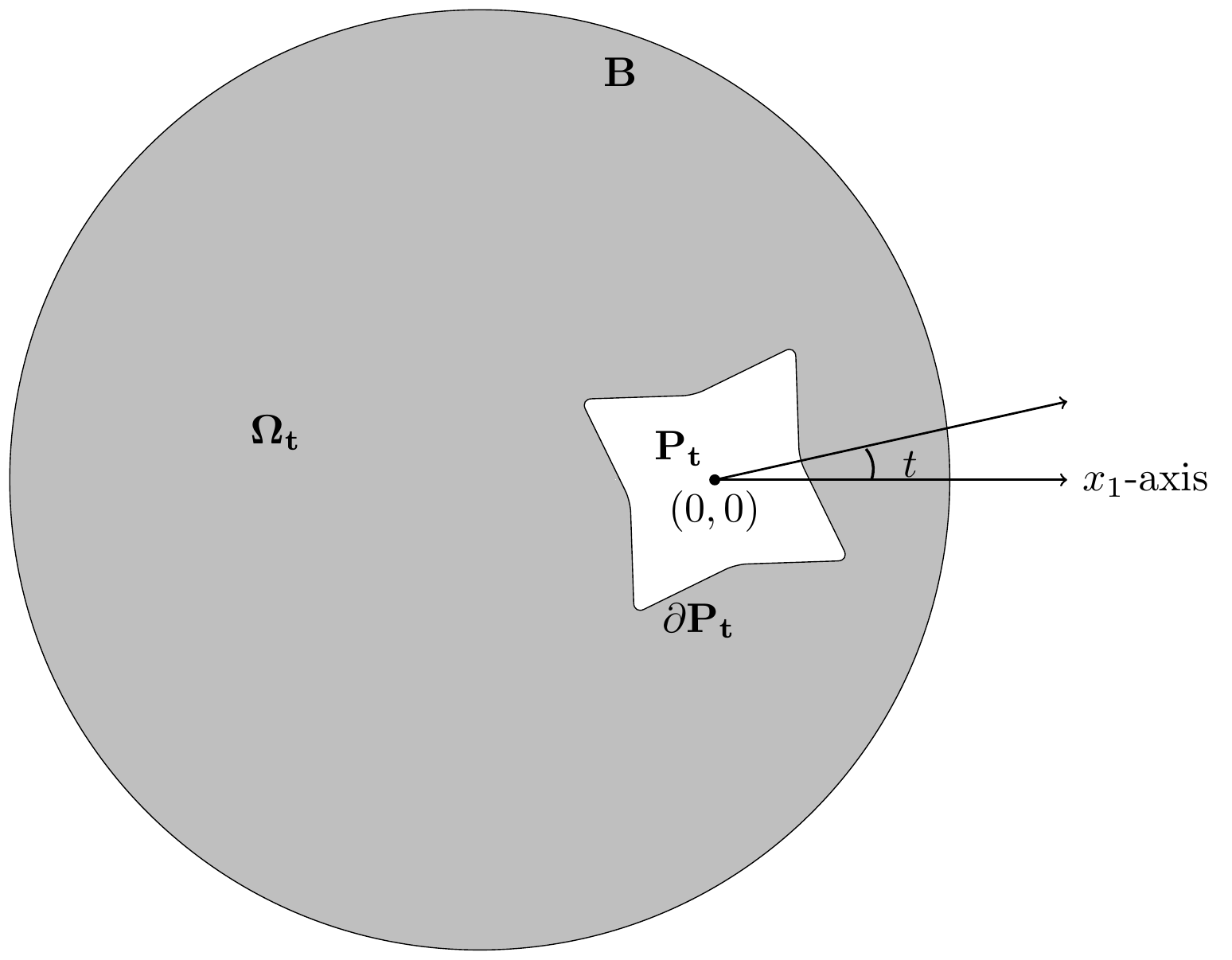}   }\hspace{8mm}
\subfloat{\includegraphics[width=0.2\textwidth]{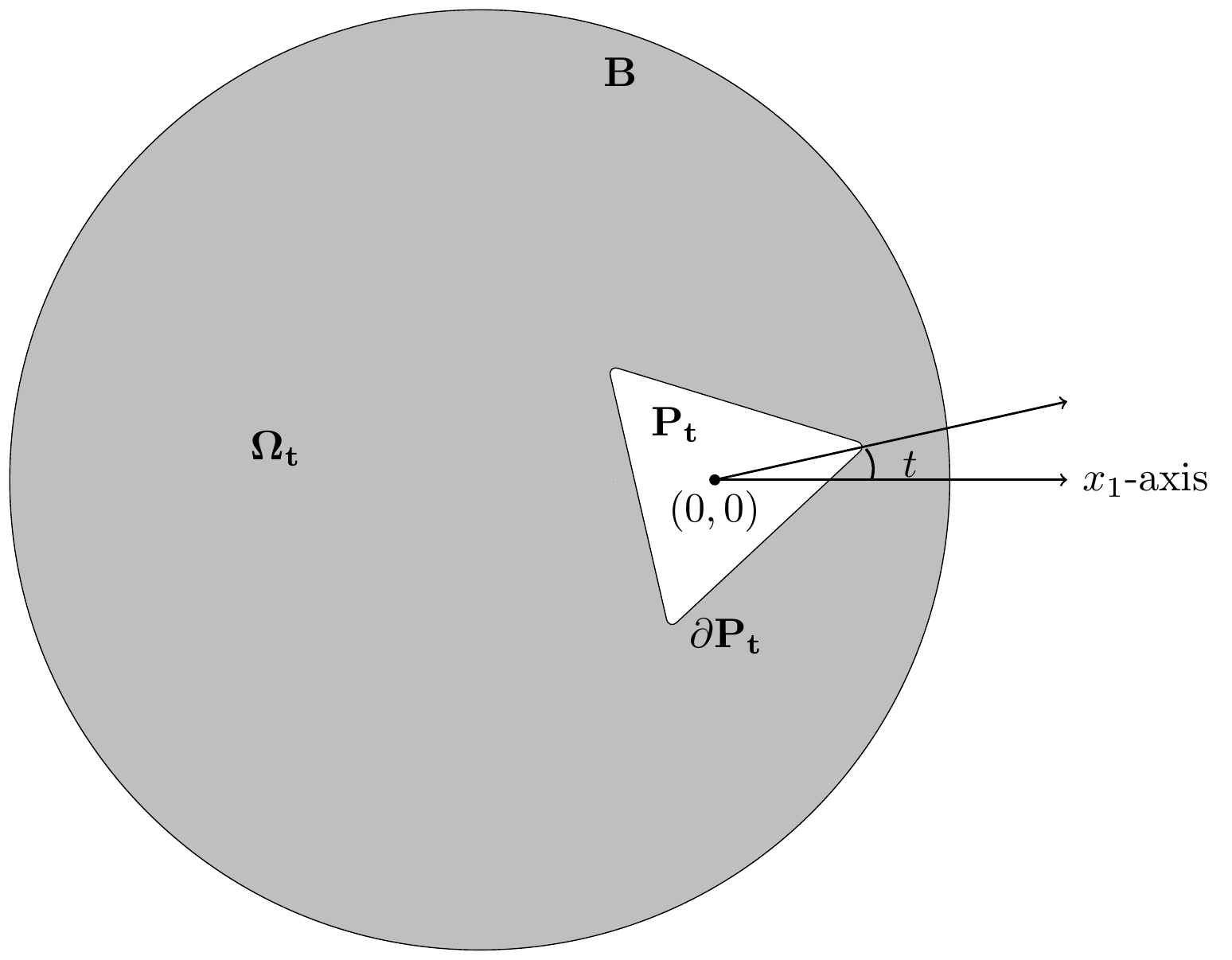}   }\hspace{8mm}
\subfloat{\includegraphics[width=0.2\textwidth]{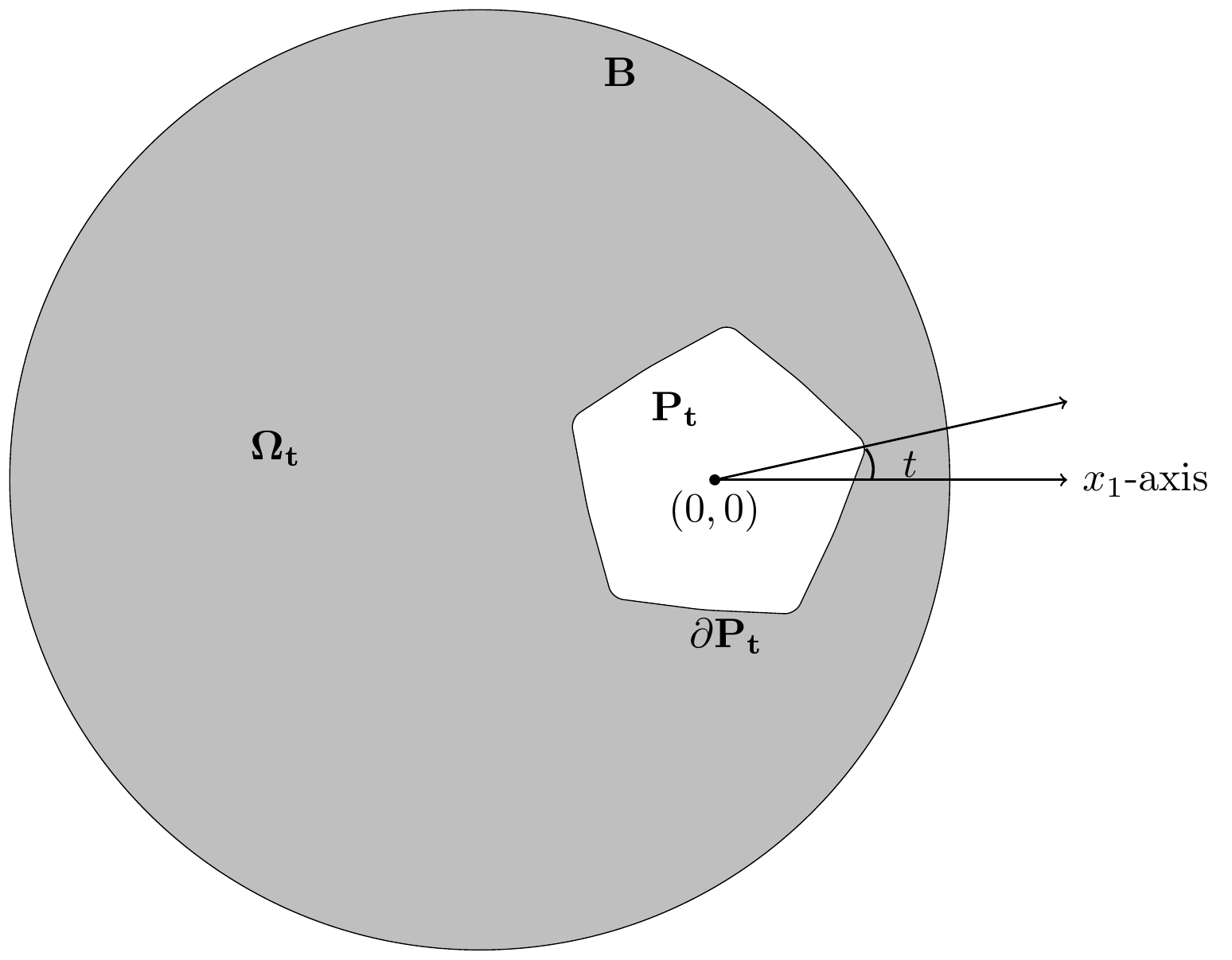}   }
\caption{Configuration at time $t$}\label{fig:t_config}
\end{figure} 

\subsubsection{\textbf{Expression for the Eulerian Derivative of the energy functional $E$}}\label{HPF} 
We recall, from section \ref{deduction}, that the Eulerian derivative $dE(\Omega, V)$ of $E$ at  
$\Omega$ in the direction $V$ exists, and is given by (\ref{EulerianDerivative}). 
In other words, the derivative $E^\prime(t)$ of $E$ at a point $t \in \mathbb{R}$ is given by 
\begin{equation}\label{Hadamard}
E^\prime(t) =- 2\int_{x \in \partial P_t} \left|{\dfrac{\partial y(t)(x)}{\partial \eta_t}}\right|^2 ~ \left<\eta_t, v\right>(x)~d\Sigma(x),
\end{equation}
where $d\Sigma$ is the line element on $\partial P_t$, $\eta_t(x)$ is the outward unit normal vector to $\Omega_t$ at $x \in \partial\Omega_t$, and $v \in \mathcal{C}_0^\infty(\Omega_t)$ is the 
 deformation vector field defined as  
$v(\zeta) = \rho(\zeta) \, \textbf{i}\zeta, \, \forall \, \zeta \in  \mathbb{C} \cong \mathbb{R}^2$.
Here, $\rho: \mathbb{R}^2 \rightarrow [0,1]$ is a smooth function with compact support in $B$ such that $\rho \equiv 1$ in a neighbourhood of $cl(conv(C_2(P)))$. 

\subsubsection{\textbf{The Energy Functional $E$ is an Even and Periodic Function with Period $\frac{2\pi}{n}$}}\label{PeriodicEnergy}
{ Recall that $n \geq 3$ is a fixed positive integer}. 
Let $R_0:\mathbb{R}^2 \rightarrow \mathbb{R}^2$ denote the reflection in $\mathbb{R}^2$ about the $x_1$-axis. 
Then, as in \cite{Anisa-Souvik}, we get  
$\Omega_{-t}=R_0(\Omega_t)$ for all $t \in \mathbb{R}$ and 
$\Omega_{\frac{2 \pi}{n}+t} = \Omega_t$ for all $t \in \mathbb{R}$. Further, since the boundary data is constant on each boundary component, the solution  $y$ 
of (\ref{laplace_equation_lc}) satisfies $y(\Omega_{\frac{2 \pi}{n}+t})= y(\Omega_t)$ for all $t \in \mathbb{R}$.
This implies that $E:  \mathbb{R} \rightarrow (0, \infty)$ is an even and periodic function with period $\frac{2\pi}{n}$. Thus we have, 
\begin{equation}\label{even_lambda}
E\left({t+\dfrac{2\pi}{n}}\right)=E(t),~ \mbox{ and }~ E(-t)=E(t) ~~~~ \forall \; t ~ \in\mathbb{R}.
\end{equation}
Therefore, it suffices to study the behaviour of $E(t)$ only on the interval $\left[{0,\frac{\pi}{n}}\right]$.  
\subsubsection{\textbf{Sufficient Condition for the Critical Points of $E$}}
The following theorem states a sufficient condition for the critical points of the function $E: \mathbb{R} \rightarrow (0, \infty) $.
\begin{proposition}[Sufficient condition for critical points of $E$]\label{critical_points}{  Let $n \geq 3$ be a fixed integer}. Then, for each $k = 0, 1, 2, \ldots,2 n-1$, $E^\prime\left(k\frac{\pi}{n}\right)=0$. 
\end{proposition}
\begin{proof} \quad
Fix $k \in \{0, 1, 2, \ldots,2 n-1\}$. Let $t_k:=k\frac{\pi}{n}$. 
Then, the domain $\Omega_{t_k}$ is symmetric with respect to the $x_1$ axis. And, because of the nice boundary conditions in (\ref{laplace_equation_lc}), the solution  $y\left(t_k \right)$ satisfies
$u\circ R_0=u$,
where $R_0$ is as in section \ref{PeriodicEnergy}. 
The proof now follows, 
as in \cite{Anisa-Souvik}, from property $(iii)$ of Lemma 4.2 of \cite{Anisa-Souvik} and the expression expression (\ref{Hadamard}) by observing that 
%
$\dfrac{\partial \left(y\left(t_k \right) \circ R_0\right)}{\partial \eta}(x)=\dfrac{\partial \left(y\left(t_k\right)\right) }{\partial \eta}\left(R_0(x)\right)$ 
for each $x$ on $\partial P_{t_k}$ where the normal derivative makes sense. 
\qed
\end{proof}
\subsection{\textbf{The Sectors of $\Omega_t$}} We recall the definition of sectors of $\Omega_t$ from \cite{Anisa-Souvik}. 
{ Fix $n \geq 3$, $n \in \mathbb{N}$.} For a fixed $t \in \mathbb{R}$ and $a, b \in \mathbb{Z}$, $a< b$, let $\sigma_{(a,b)}:= \sigma\left( {t+\frac{a\pi}{n},t+\frac{b\pi}{n}}\right):=\left\{ r \, e^{i \phi} \in  \mathbb{R}^2
\, : \,  \phi \in \left(t+\frac{a \pi}{n},t+\frac{b \pi}{n}\right), \, r \in \mathbb{R} \right\}$. 
Then, as in \cite{Anisa-Souvik}, equation (\ref{Hadamard}) can be written as
\begin{equation}\label{sum_integral}
\begin{aligned}
E^\prime(t) = &-2\sum_{k=0}^{n-1} \int_{\partial P_t \cap \sigma_{(k,k+1)}} \left|{\dfrac{\partial y(t)(x)}{\partial \eta_t}}\right|^2 ~ \left< \eta_t,  v \right>(x)~d\Sigma(x)\\
& -2\sum_{k=n}^{2n-1} \int_{\partial P_t \cap \sigma_{(k,k+1)}} \left|{\dfrac{\partial y(t)(x)}{\partial \eta_t}}\right|^2~ \left< \eta_t, v \right>(x)~d\Sigma(x).
\end{aligned}
\end{equation}
%
Please refer to Figure 7 of \cite{Anisa-Souvik} for a pictorial illustration of sectors of $\Omega_t$. 
\subsection{\textbf{A Sector Reflection Technique}}
 Here onwards, we fix $n\geq 3$, $n \in \mathbb{N}$, $n$ even. 
For $\alpha \in [0, 2 \pi]$, the set $z_\alpha := \{ r e^{i \alpha} \,| \, r \in \mathbb{R}\}$ denotes the line in $\mathbb{R}^2$ corresponding to angle $\phi=\alpha$, represented in polar coordinates. Let $R_{\alpha}:\mathbb{R}^2 \rightarrow \mathbb{R}^2$, $\alpha \in \mathbb{R}$, denote the reflection map about the $z_{\alpha}$-axis. For each $t \in \mathbb{R}$, the obstacle $P_t$ is symmetric with respect to the line $z_{t+\frac{(k+1)\pi}{n}}$.  We have, for $k=0,1,2, \ldots,2 n-1$, 
\begin{equation} \label{sectorinclusion}
\displaystyle{R_{t+\frac{(k+1)\pi}{n}}(\partial P_t \cap \sigma_{(k, k+1)}) =\partial P_t \cap \sigma_{( k+1, k+2)}}.
\end{equation}
For $k =0, 1,2, \dots, 2n-1$, let $H_1^k(t):= \Omega_t \cap \sigma_{(k,k+1)}$. Now, let $\tilde{H}_1^k :=cl(\Omega_t) \cap \sigma_{(k,k+1)} $. 
%
\subsection{\textbf{The Rotating Plane Method}}
\noindent { Recall here that $n \geq 3$ is a fixed even natural number.} 
As in \cite{Anisa-Souvik} we have the following: \\
 For each $k =0,2,4, \ldots , n-2$,\begin{equation}\label{ref_1}
\begin{aligned}
&\int_{\partial P_t \cap \sigma_{(k,k+1)}} \left|{\dfrac{\partial y(t)}{\partial \eta_t}}(x)\right|^2~\left< \eta_t,  v\right>(x)~d\Sigma + \int_{\partial P_t \cap \sigma_{(k+1,k+2)}} \left|{\dfrac{\partial y(t)}{\partial \eta_t}}(x)\right|^2~ \left<\eta_t , v\right>(x)~d\Sigma\\
=&\int_{\partial P_t \cap \sigma_{(k,k+1)}} \left({\left|{\dfrac{\partial y(t)}{\partial \eta_t}}(x)\right|^2-\left|{\dfrac{\partial y(t)}{\partial \eta_t}}(x^\prime)\right|^2}\right)~ \left<\eta_t,  v \right>(x)~d\Sigma.
\end{aligned}
\end{equation}
\begin{equation}\label {inrprdctsgn1}\left< \eta_t,  v \right>  >0 ~ \mbox{ on }\partial P_t \cap \sigma_{(k,k+1)}~\mbox{ for each } k =0,2,4, \ldots , n-2.
\end{equation}
For each $k =n, n+2, \ldots, 2n-2$,
\begin{equation}\label{ref_2}
\begin{aligned}
&\int_{\partial P_t \cap \sigma_{(k,k+1)}} \left|{\dfrac{\partial y(t)}{\partial \eta_t}}(x)\right|^2~\left< \eta_t,  v\right>(x)~d\Sigma + \int_{\partial P_t \cap \sigma_{(k+1,k+2)}} \left|{\dfrac{\partial y(t)}{\partial \eta_t}}(x)\right|^2~ \left<\eta_t , v\right>(x)~d\Sigma\\
=&\int_{\partial P_t \cap \sigma_{(k+1,k+2)}} \left({\left|{\dfrac{\partial y(t)}{\partial \eta_t}}(x)\right|^2-\left|{\dfrac{\partial y(t)}{\partial \eta_t}}(x^\prime)\right|^2}\right)~ \left<\eta_t,  v \right>(x)~d\Sigma,
\end{aligned}
\end{equation}
\begin{equation}\label {inrprdctsgn2}\left< \eta_t,  v \right>  >0 ~ \mbox{ on }\partial P_t \cap \sigma_{(k+1,k+2)} ~\mbox{ for each } k =n, n+2, \ldots , 2n-2.\end{equation}
Here, $x^\prime:= R_{t+\frac{(k+1)\pi}{n}}(x)$.
\subsection{\textbf{Necessary Condition for the Critical Points of $E$}}
{ Recall here that $n \geq 3$ is a fixed even integer.} We finally show that $\left\{\frac{k \pi}{n}\, |\, k =0,1, \ldots n-1\right\}$ are the only critical points of $E$, and that, between every pair of consecutive critical points of $E$, it is a strictly monotonic function of the argument. In view of Proposition \ref{critical_points} and equation (\ref{even_lambda}), it now suffices to study the behaviour of $E$ only on the interval $\left({0,\frac{\pi}{n}}\right)$.   
\begin{proposition}[Necessary condition for critical points]\label{complete_critical_points}{  Fix $n \geq 3$, $n$ even, $n \in \mathbb{N}$.} Then, for each $t\in ]0,\frac{\pi}{n}[$, $E^\prime(t) < 
0$.
\end{proposition}
\begin{proof} \quad  Notice here that if $u$ is a solution of (\ref{laplace_equation_lc}) for $g=M>0$, then $-u$ will be a solution of (\ref{laplace_equation_lc}) for $g=-M$. And hence, the energy for the case $g=M>0$ and $g=-M<0$ are the same. Therefore, it is enough to consider the case $g=M>0$.

Fix $t\in ]0,\frac{\pi}{n}[$.
Using (\ref{ref_1}) and (\ref{ref_2}), integral (\ref{sum_integral}) can be written as 

\begin{equation}\label{final_integralnew}
\begin{aligned}
E^\prime(t) 
 &= -2\sum_{\substack{0\leq k\leq n-2\\k \mbox{ \small{even} }}}
\int_{\partial P_t \cap \sigma_{(k,k+1)}} \left({\left|{\dfrac{\partial y(t)(x)}{\partial \eta_t}}\right|^2-\left|{\dfrac{\partial y(t)(x^\prime)}{\partial \eta_t}}\right|^2}\right)~ \left<\eta_t, v\right> (x)~d\Sigma(x)\\ &-2 \sum_{\substack{n\leq k\leq 2n-2\\k \mbox{ \small{even} }}} \int_{\partial P_t \cap \sigma_{(k+1,k+2)}} \left({\left|{\dfrac{\partial y(t)(x)}{\partial \eta_t}}\right|^2-\left|{\dfrac{\partial y(t)(x^\prime)}{\partial \eta_t}}\right|^2}\right)~ \left<\eta_t, v \right> (x)~d\Sigma(x) \\
\end{aligned}
\end{equation}
Let $\displaystyle H(t):= \bigcup_{\substack{0\leq k\leq n-2\\k \mbox{ \small{even} }}}
H_1^k(t)$. Let $w(x) := y(t)(x) - y(t)(x^\prime)$. By Lemma 6.1 of \cite{Anisa-Souvik}, 
the real valued function $w$ is well-defined on $H(t)$. Moreover, $w \equiv 0$ on $\partial P_t \cap \partial H(t)$ and also on $\partial H(t) \cap z_{t+ k \frac{\pi}{n}}$ for each\\ $k=1, 3, \ldots n-1$.  That is, $
w(x) = 0 ~ \forall \; x\in \partial{H(t)}\bigcap \left({\partial{P_t}\bigcup_{\substack{1\leq k\leq n-1\\k \mbox{ \small{odd} }}}z_{t+\frac{k\pi}{n}}}\right)$. Moreover, since $y(t) = M >0$ on $\partial{B}$ and $y(t) \in (0, M)$ 
 inside $\Omega(t)$, and since for each $k=0, 2, \ldots n-2$,  the reflection of $\partial{H_1^k(t)}\cap \partial B$ about the axis $z_{t+( k+1) \frac{\pi}{n}}$ lies completely inside $H_1^{k+1}(t) \subset \Omega(t)$, we obtain $w(x) > 0$ for each $x$ in $\left( \partial{H(t)}\cap \partial B \right) \setminus \left(\bigcup_{\substack{1\leq k\leq n-1\\k \mbox{ \small{odd} }}} z_{t+\frac{k\pi}{n}}\right)$. 
Now, with arguments similar to the ones in the proof of Proposition 6.2 of \cite{Anisa-Souvik}, it can be shown that 
$w(x) > 0 ~\forall \; x \in  \partial H(t) \bigcap \bigcup_{\substack{0 \leq k\leq n-2\\k \mbox{ \small{even} }}} z_{t+\frac{k\pi}{n}}$.
This is equivalent to saying that for each $k$, $0 \leq k \leq n-2, k$ even, $w(x) > 0$ for all $x \in  \partial H_1^k(t) \cap z_{t+\frac{k\pi}{n}}$. 
Therefore, the non-constant function $w$ satisfies 
$-\Delta w =0
 \mbox{ in } H(t), ~
w  \geq 0,  \mbox{ on } \partial H(t)$.
Hence, by the maximum principle, $w$ is non-negative on the whole of $H(t)$. 
Since $w$ achieves its minimal value zero on $\bigcup_{\substack{0\leq k \leq n-2\\k\equiv 0\bmod 2}}\left({\partial P_t \cap \sigma_{(k,k+1)}}\right)$ $\subset \partial H(t)$, by the Hopf maximum principle, one has $
\dfrac{\partial w}{\partial \eta_t}(x) 
{ <}0 ~~ \forall\; x\in \bigcup_{\substack{0\leq k \leq n-2\\k\equiv 0\bmod 2}}\left({\partial P_t \cap \sigma_{(k,k+1)}}\right)$. 
Also, by the application of the Hopf maximum principle to problem (\ref{laplace_equation_lc}) for $g=M>0$, it follows that
 $$\dfrac{\partial y(t)}{\partial \eta_t}(x) <0 ~\forall\; x \in \partial P_t, ~~ \mbox{ and } ~~\dfrac{\partial y(t)}{\partial \eta_t}(x) >0 ~\forall\; x \in \partial B.$$ Thus,
\begin{equation}\label {drvsgn1}
\left|{\dfrac{\partial y(t)}{\partial \eta_t}(x)}\right|^2 -\left|{\dfrac{\partial y(t)}{\partial \eta_t}(x^\prime)}\right|^2  >~ 0 ~ \forall \; x \in \bigcup_{\substack{0\leq k \leq n-2\\k\equiv 0\bmod 2}}\left({\partial P_t \cap \sigma_{(k,k+1)}}\right).\end{equation}
Now, from (\ref{drvsgn1}) and (\ref{inrprdctsgn1}), it follows that the first term in (\ref{final_integralnew}) is strictly negative. Similarly, one can prove using (\ref{inrprdctsgn2}) that the second term in (\ref{final_integralnew}) is also strictly negative. 
This proves the proposition for $n$ even. 
\qed
\end{proof}
\subsection{\textbf{Proof of Theorem \ref{max_min}}
} Theorem \ref{max_min}, for each even natural number $n$, $n \geq 3$, now follows from Propositions \ref{critical_points}, 
\ref{complete_critical_points}, and equation (\ref{even_lambda}).

\subsection{\textbf{The $n$ Odd Case}}
{ 
We face exactly the same difficulties as in \cite{Anisa-Souvik} in characterizing the optimal configurations for the odd order case completely. Please refer to section 6.7 of \cite{Anisa-Souvik} for details.
Nevertheless, we provide some numerical evidence that enables us to make a conjecture that Theorem \ref{max_min} holds true for $n$ odd too. 
\section{Proof of Theorem \ref{global}}
\label{secGlobal}

Let $d\geq 0$ denote the distance between the center of the disk and the center of the obstacle having the dihedral symmetry. Let $t\in[0,2\pi[$ denote the angle by which the obstacle is rotated about its center in the anticlockwise direction starting from the initial configuration as described in Sec. 5.1.1. clearly, $E$ is a function of both $d$ and $t$. Please recall that when $d=0$, the map $t \mapsto E(t)$ is constant as the domains remain isometric and because of the suitable boundary conditions the solution  remains unaltered. For a fixed $d>0$, it is interesting to study the behaviour of the map $t \mapsto E(t)$. This is what we have studied in Theorem \ref{max_min}.

 It also makes sense to study the behaviour of the map $d \mapsto E(d)$ for a fixed $t\in [0,2\pi[$. By arguments similar to the proof of Theorem 2.1 on page 244 of \cite{harrel_kurata} we get the following result with just the assumption that the set P is convex as well as piecewise smooth and that it be reflection-symmetric about some
line $L$. 

\begin{theorem}\label{provedinHarrell} Assume that $\Omega$ has the interior reflection property with respect to a line $L$ about which the set $P$ is reflection-symmetric. Suppose that $P$ is translated in the direction of a unit vector $v$ perpendicular to $L$ and pointing from the small side to the big side. Then, $\dfrac{dE}{dv} <0$.
\end{theorem}

From Theorem \ref{provedinHarrell} we get the following result: 
\begin{corollary}\label{notproved} Fix $n \geq 3$, $n \in \mathbb{N}$. Fix $t \in [0,2\pi[$.  
Let $x$ denote the center of the obstacle $P_{(d,t_0)}$. Then, at any minimizing $x$, 
\begin{enumerate}[a)]
\item $\Omega=B \setminus P_{(d,t_0)}$ has no hyperplane of interior reflection containing $x$. Moreover, at any maximizing $x$, either statement (a) above is true, or else
\item the circumcircle $C_2$ of $P_{(d,t_0)}$ intersects the small side of $\partial B$. 
\end{enumerate}
\end{corollary}
Clearly, the domain $\Omega_{(d,t_0)}$ enjoys the interior reflection property w.r.t. its all axes of symmetry of $P_{(d,t_0)}$ that are not the diameters of $B$. Therefore, as in \cite{harrel_kurata}, we immediately get 
 that
\begin{corollary} \label{MC}  Fix $n \geq 3$, $n \in \mathbb{N}$ and $t \in [0,2\pi[$. Then,  (a) the concentric configuration, i.e., $d=0$, is the only candidate for the minimiser of the map $d \longmapsto E(d)$, and 
(b) at any maximising configuration of the map $d \longmapsto E(d)$, the circumcircle $C_2$ of the obstacle $P_{(d,t_0)}$ must intersect $\partial B$.
\end{corollary} Since $\Delta$ is invariant under the isometries of the domain, and since the boundary data in (\ref{laplace_equation_lc}) is nice, it follows that for a fixed $t\in[0,2\pi[$, in order to study the behaviour of $d\mapsto E(d)$, it is enough to translate the center of the obstacle $P_t$ along the positive $x_1$-axis. Since $E$ is an even function, 
it follows by Theorem \ref{provedinHarrell} that, $d\mapsto E(d)$ is minimum for $d=0$ and is a strictly increasing function of $d$ in $]0, r_1-r_0^2[$. Then, Corollary \ref{MC} along with Theorem \ref{max_min} imply Theorem \ref{global} that characterises the maximising and the minimising configurations over the family of domains $\mathcal{G}$. Applying the idea from \cite{harrel_kurata} to the candidates for the minimising configurations over $\mathcal{G}$ for $n$ even, $n \geq 3$, we get that, when $g=M>0$, at the global maximising configurations, w.r.t. both the translations of the obstacle within $B$ as well as the rotations of the obstacle about its center, the obstacle must be in an OFF position w.r.t. $B$ with its outer vertex touching $\partial B$. The global minimiser for $n \geq 3$, even or odd, remains to be the concentric configuration.
\section{Proof of Theorem \ref{expansion}}
\label{secExpansion}
In order to study the behaviour of $E$ w.r.t. the expansion or contraction of the obstacle, we fix $d \geq 0$ and $t \in [0, 2 \pi)$ such that $P_{(d,t)}$ lies completely inside $B$. We now scale $P=P_{(d,t)}$ by a factor $\lambda > 0$, $\lambda \in \mathbb{R}$, such that $\lambda \, P_{(d, t)} := \{\lambda \, x \,|\, x \in P_{(d, t)} \}$ still lies completely inside $B$. Let $\Omega_{\lambda}:= B \setminus \lambda \, P_{(d,t)}$. 
Let $E(\lambda):=E(\Omega_{\lambda})$. Let $\mathcal{H}$ be defined as $\{\Omega_{\lambda} \, :\, \lambda >0, \lambda \in \mathbb{R}, \lambda \, P \subset B\}$. We denote the largest admissible value of $\lambda$ by $\lambda_P$ which depends on $d$ and $t$ both.

 Fix $n \in \mathbb{N}, n\geq 3$. Fix $\lambda>0$. Here, the vector field $V$ which is responsible for the perturbation we wish to establish, is given by $V(x)= \lambda\, \rho(x)\, x$. Here, as before, $\rho: \mathbb{R}^2 \rightarrow [0,1]$ is a smooth function with compact support in $B$ such that $\rho \equiv 1$ in a neighbourhood of $cl(conv(C_2(P)))$.  As in section \ref{scal}, one can prove that the Eulerian derivative of $E$ at  
$\Omega_{\lambda}$ in the direction $V$ exists for each $\lambda >0$. Further, for each $\lambda >0$, as in equation (\ref{EulerianDerivativeS1}), we get 
 \begin{equation}\label{Eprimelambda}E^\prime(\lambda)= -2\int_{\partial P} \left( \dfrac{\partial u}{\partial n}\right)^2 \left<V, n\right>d\Sigma =-2 \, \lambda \, \int_{\partial P} \left( \dfrac{\partial u}{\partial n}\right)^2(x) \left<x, n(x)\right>d\Sigma .\end{equation}
 Now, from (i) of Lemma 4.2 of \cite{Anisa-Souvik} we have, for $x= f(\phi) e^{i \phi}$ on $\partial P$, $n(x) = \dfrac{i \, f^\prime(\phi) e^{i \phi} - f(\phi) e^{i\, \phi}}{\sqrt{f^2(\phi)+ [f^\prime(\phi)]^2}}$. As a result we get, $\left<x, n(x)\right> = \dfrac{ - f^2(\phi)}{\sqrt{f^2(\phi)+ [f^\prime(\phi)]^2}} = \dfrac{ -\|x\|^2}{\sqrt{f^2(\phi)+ [f^\prime(\phi)]^2}} $. Substituting this in (\ref{Eprimelambda}) we get, 
 $\displaystyle E^\prime(\lambda)= 2 \, \lambda \, \int_{\partial P} \left( \dfrac{\partial u}{\partial n}\right)^2(x) \dfrac{ \|x\|^2}{\sqrt{f^2(\phi)+ [f^\prime(\phi)]^2}} \,d\Sigma >0.$ Thus we have proved that, the map $\lambda \longmapsto E(\lambda)$ is a differentiable function from $(0, \lambda_P)$ to $(0, \infty)$ and that the energy functional is a strictly monotonically increasing function of $\lambda$ with $E^\prime(\lambda) \longrightarrow 0$ as $\lambda \longrightarrow 0$. This proves Theorem \ref{expansion}. 
\section{Numerical Results} \label{sec:num_results} We give some numerical evidence supporting Theorem \ref{max_min}.  
We first fix the value of the boundary data $M$ to be 1. We provide numerical evidence for $n=5$ and conjecture that Proposition \ref{complete_critical_points}, and hence, Theorem \ref{max_min} hold true for $n$ odd too. Therefore, we choose two obstacles, a square and a pentagon, with dihedral symmetry of order $n=4,5$. We solve the boundary value problem (\ref{laplace_equation_lc}) in the domain $\Omega= B \setminus P$ using finite element method with $P^1$ elements (see e.g., \cite{SR1,SR2}) on a mesh with element size $h=0.018$. The result for the square obstacle are shown in Figure \ref{simulation_results_sq}.
\begin{figure}[H]
  \centering
    \subfloat[      OFF position]{%
     \label{fig:min} \includegraphics[scale=0.2,keepaspectratio]{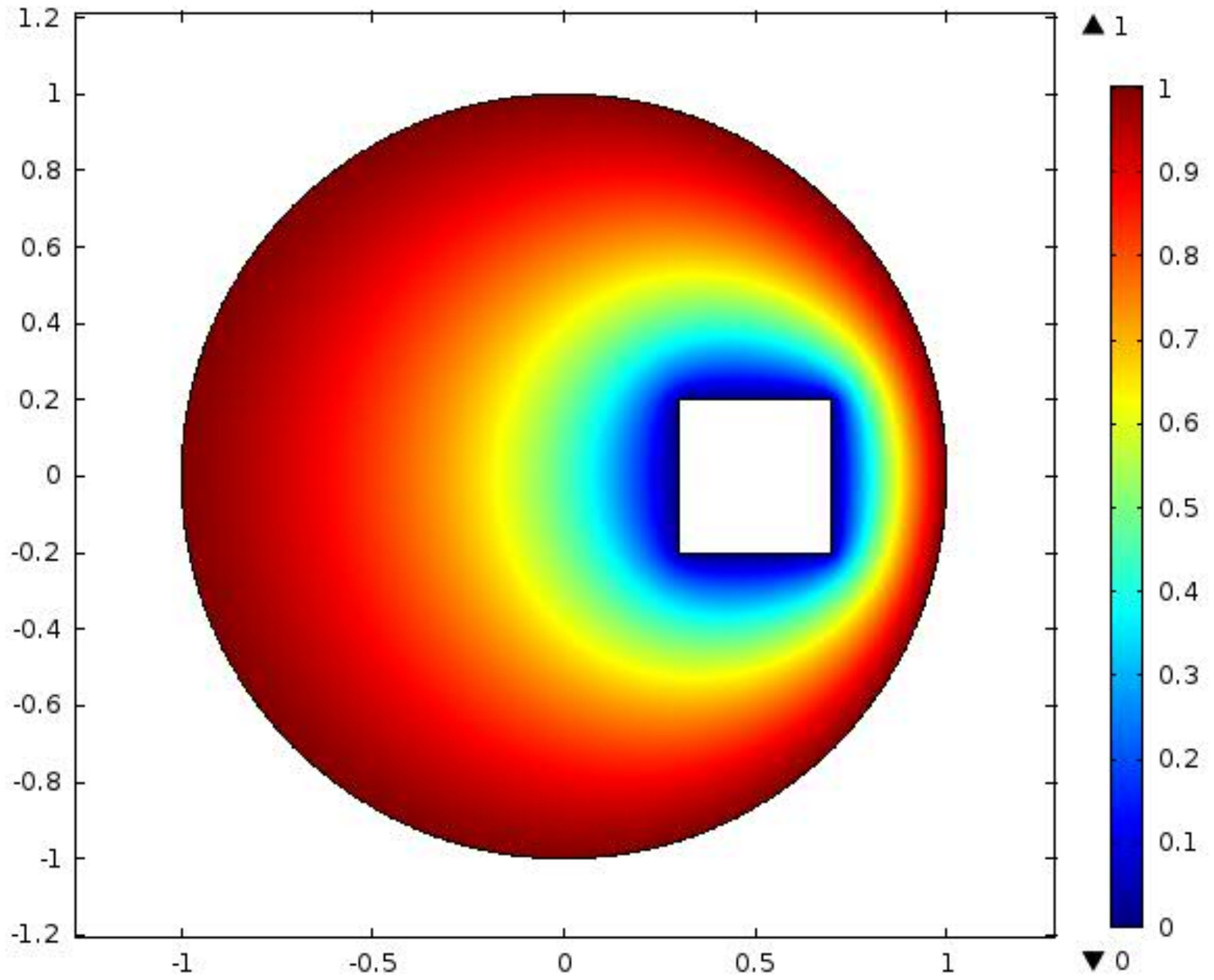}}
    \subfloat[      Intermediate Position]{%
      \label{fig:middle}\includegraphics[scale=0.2,keepaspectratio]{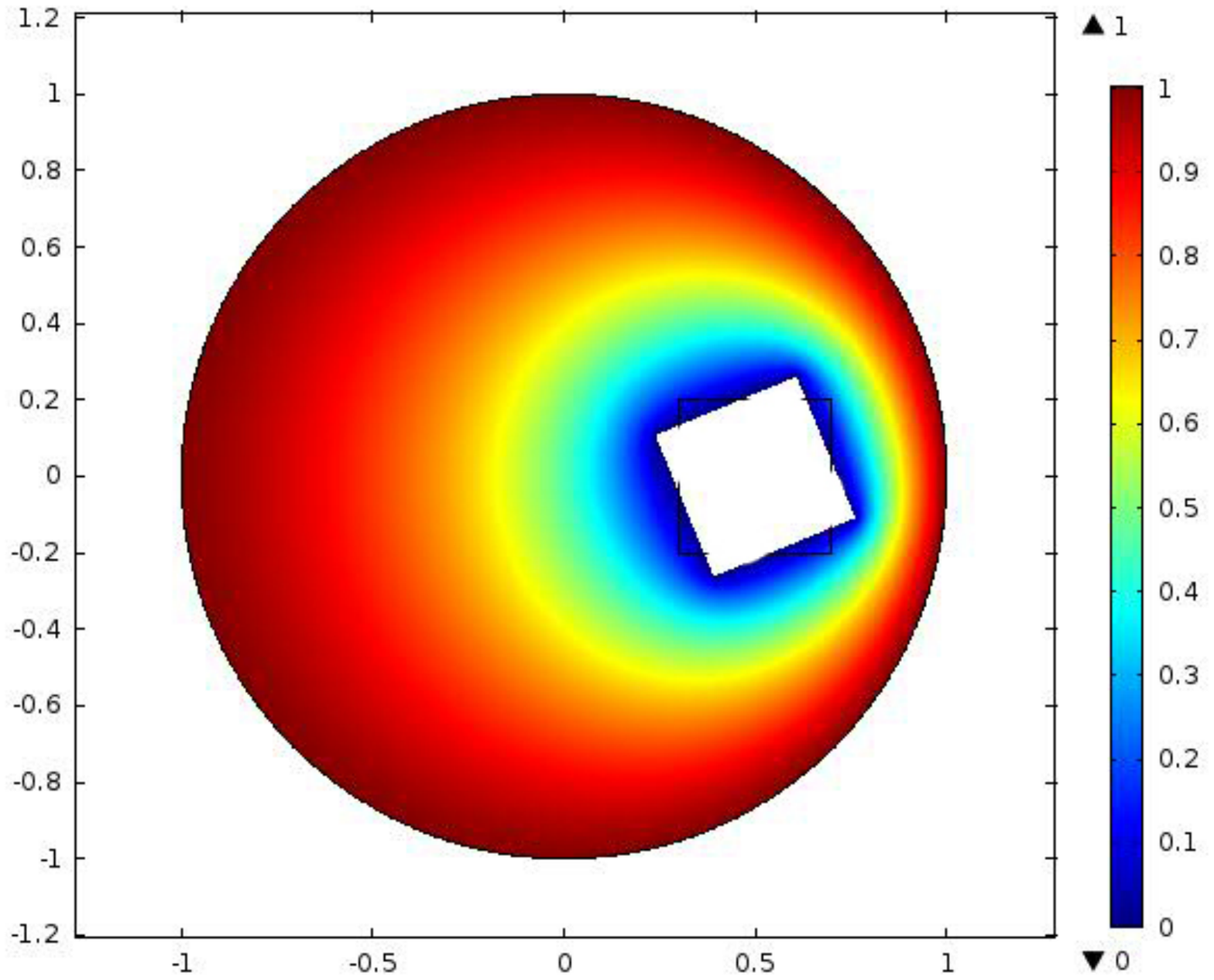}}
        \subfloat[      ON position]{%
      \label{fig:max}\includegraphics[scale=0.2,keepaspectratio]{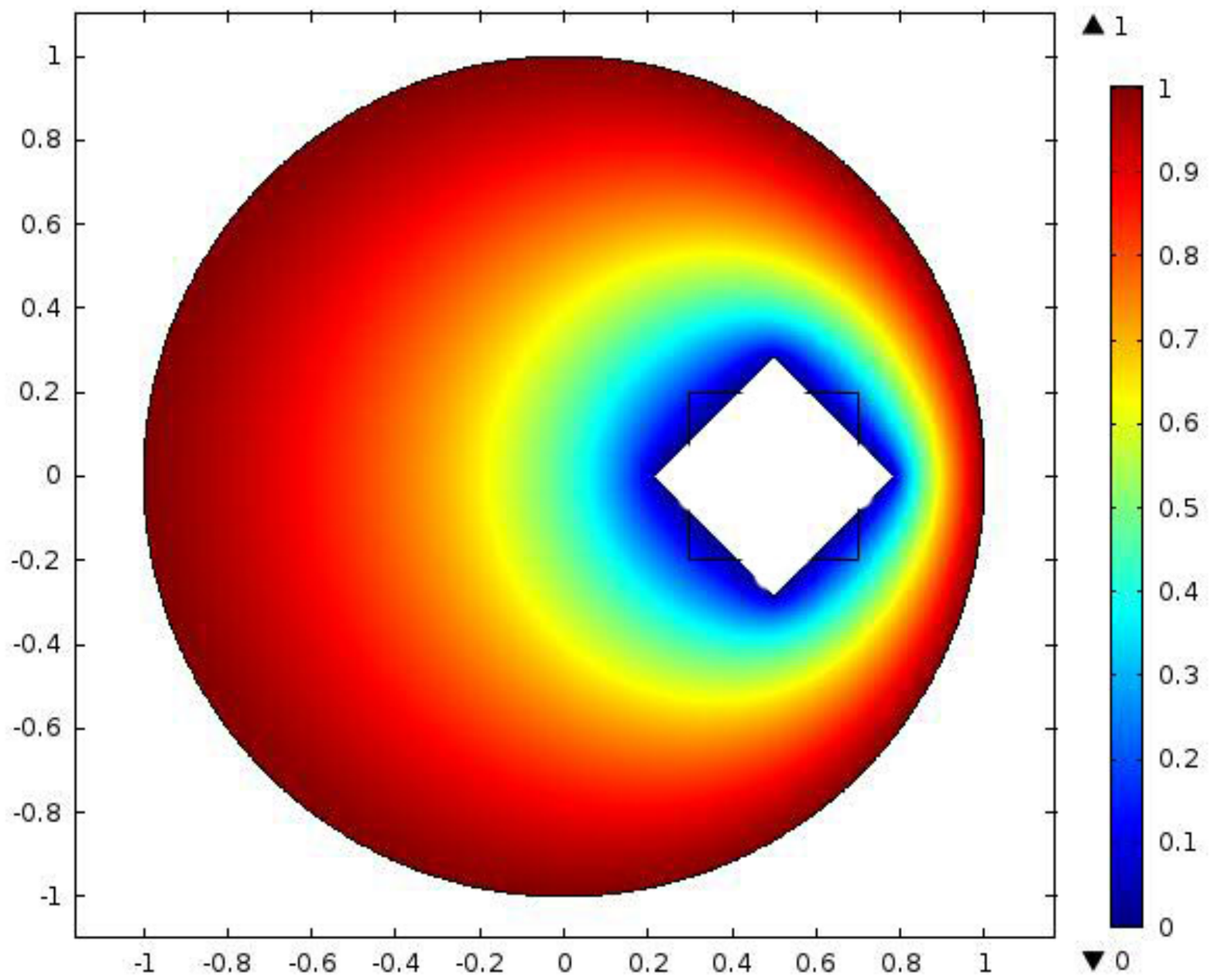}}
      \caption{Simulations of extremal configurations for a square obstacle }
       \label{simulation_results_sq}
      \end{figure}
 Figures \ref{fig:min}-\ref{fig:max} shows the OFF, intermediate and ON configurations. The OFF and the ON configurations are the maximizer and the minimizer for the energy $E$ respectively, which is also reflected in Table \ref{table1}.  
\begin{table}[H]
\caption{Values of the energy functional $E$ at different configurations for a square obstacle when $d=0.5$}
\begin{center}
\begin{tabular}{|c| c| c|c|c|c|}
\hline
Configuration &OFF &INTERMEDIATE &ON &INTERMEDIATE & OFF  \\ [0.1ex]
\hline
$\theta$ & 0 & $\pi/8$ & $\pi/4$ & 3 $\pi/4$ & $\pi/2$ \\
\hline
$E$ & 5.57787	&5.57389 &	5.56991	& 5.57386 	&5.57787\\[0.1ex]
\hline
\end{tabular}
\end{center}
\label{table1}
\end{table}
We next conjecture that Theorem \ref{max_min} is true for odd $n$ too by demonstrating quantitative and qualitative results for an obstacle having pentagonal shape. The observations for obstacle having a pentagonal shape are shown in Figure \ref{simulation_results_pentagon}.
\begin{figure}[H]
  \centering
    \subfloat[OFF position, $\theta=0$]{
     \label{fig:pentagon_min} \includegraphics[scale=0.15,keepaspectratio]{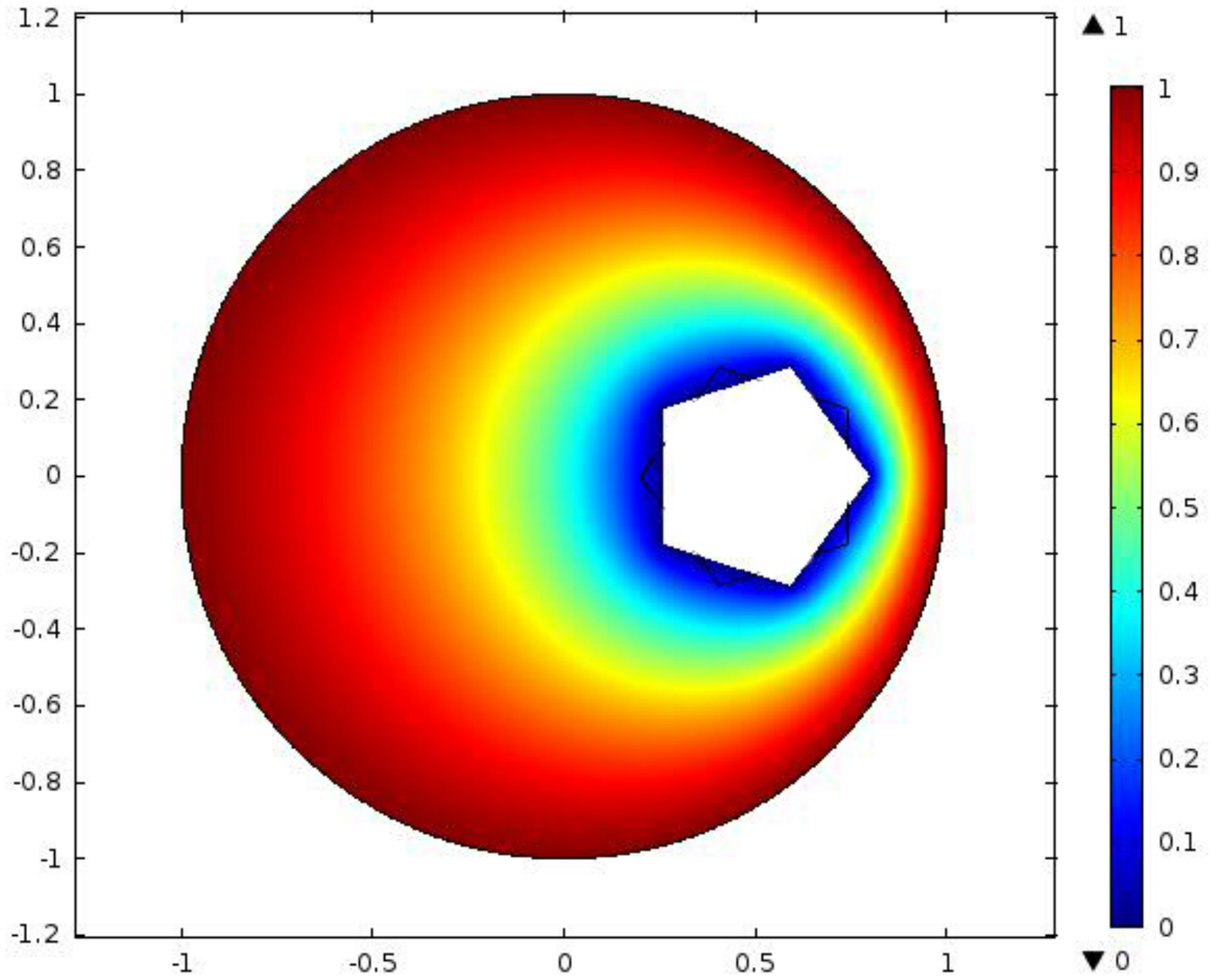}}\hspace{7mm}
    \subfloat[Intermediate Position, $\theta=\pi/10$]{
      \label{fig:pentagon_middle}\includegraphics[scale=0.15,keepaspectratio]{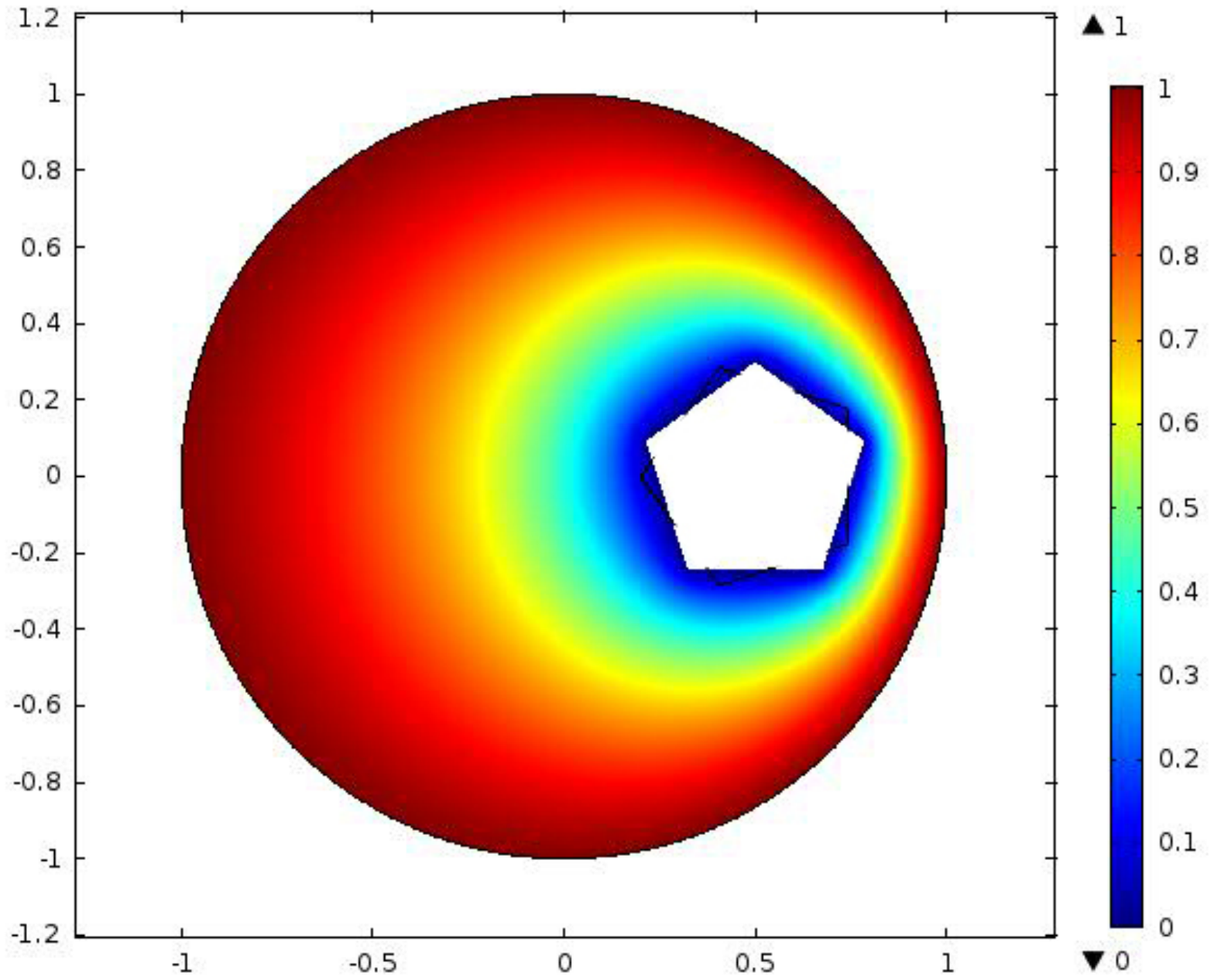}}\hspace{7mm}        \subfloat[ON position, $\theta=\pi/5$]{
      \label{fig:pentagon_max}\includegraphics[scale=0.15,keepaspectratio]{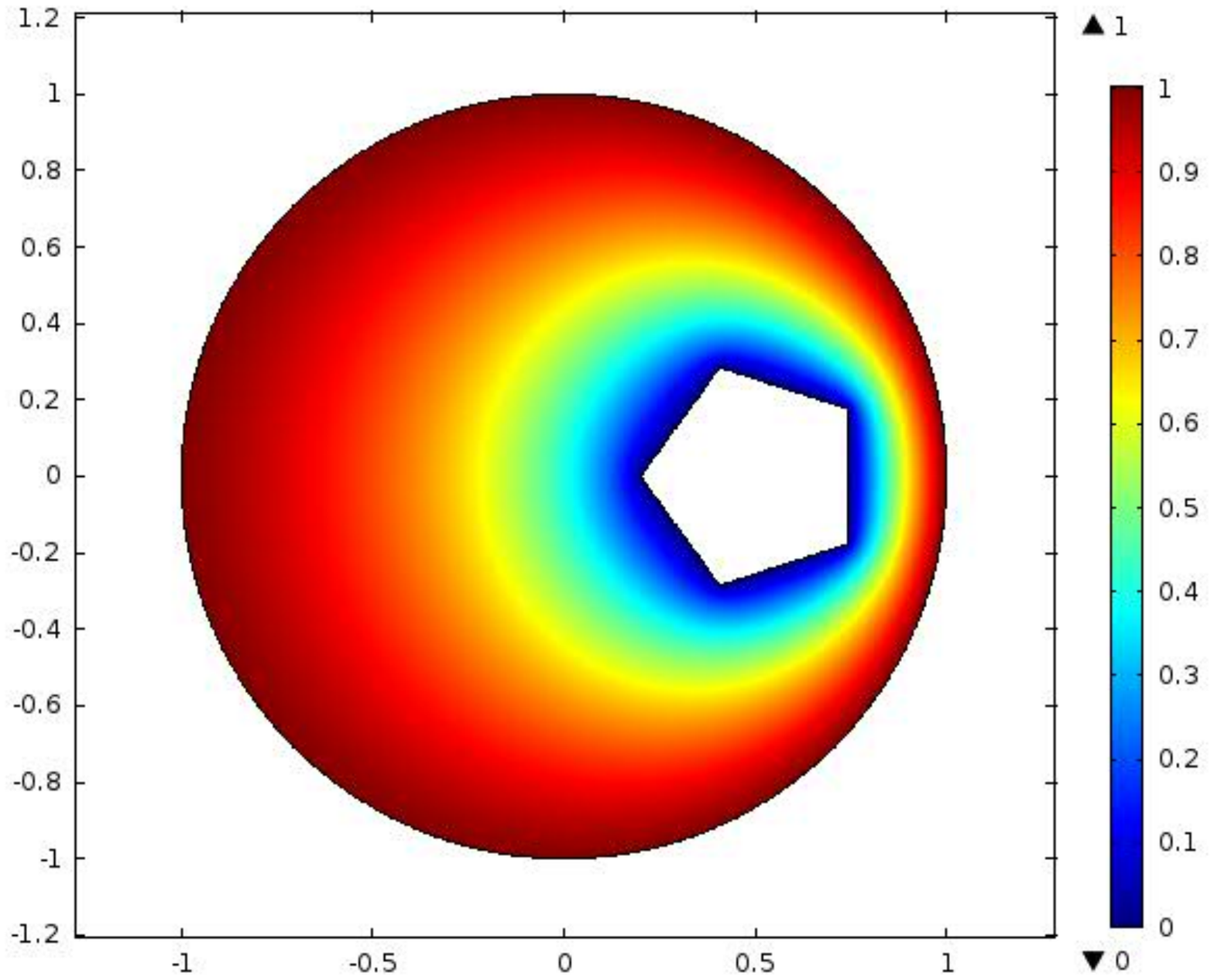}}\hspace{7mm}
      \subfloat[Intermediate position, $\theta=3 \pi/10$]{%
      \label{fig:pentagon_max}\includegraphics[scale=0.15,keepaspectratio]{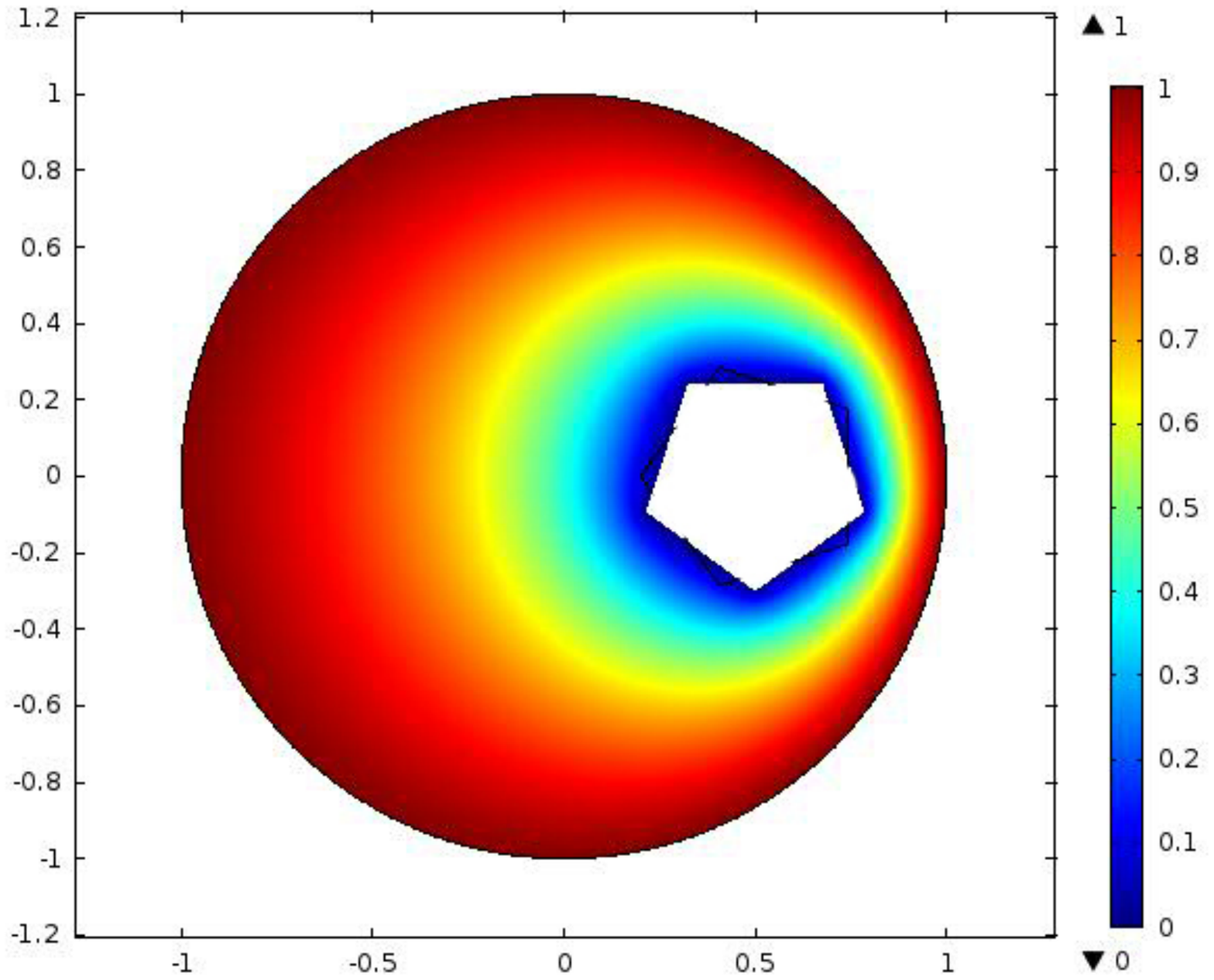}}
      \caption{Simulations of extremal configurations for a pentagonal obstacle}
      \label{simulation_results_pentagon}
 \end{figure}    
Figures \ref{fig:pentagon_min}-\ref{fig:pentagon_max} shows the OFF, intermediate and ON configurations. The OFF and the ON configurations are the maximizer and the minimizer for the energy $E$ respectively, which is also reflected in Table \ref{table2}.  
\begin{table}[H]
\caption{Values of the energy functional $E$ at different configurations for a pentagonal obstacle when $d= 0.5$}
\begin{center}
\begin{tabular}{|c| c| c|c|c|c|}
\hline
Configuration &  OFF &INTERMEDIATE &ON & INTERMEDIATE & OFF
\\ [0.1ex]
\hline
$\theta$ & 0	&$\pi/10$ & $\pi/5$ & $3 \pi/10$ & $2 \pi/5$ \\
\hline 
$E$ &6.30518 & 6.30378 &	6.30239 & 6.30378 &6.30518\\[0.1ex]
\hline
\end{tabular}
\end{center}
\label{table2}
\end{table} 
We also give some numerical evidence supporting Theorem \ref{global}.  
We once again choose two obstacles, a square and a pentagon, with dihedral symmetry of order $n=4,5$. 
Tables \ref{table3} and \ref{table4} indicates the behaviour of the energy functional w.r.t $d$ and $\theta$. Here, the values in boldface text are the ones where the obstacle is touching the boundary of the disk. The term NA stands for the cases where the $P$ starts to intersect the outside of the disk. 
\begin{table}[H]
\caption{Values of the energy functional $E(d, \theta)$ at different configurations for a square obstacle}
\begin{center}
\begin{tabular}{|c| c| c|c|c|}
\hline
& & 
OFF ($\theta =0$)
&
INTERMEDIATE ($\theta = \frac{\pi}{8}$) 
& 
ON ( $\theta = \frac{\pi}{4}$ )
 \\ [1.5ex]
\hline
& $d$ & $E({d, 0})$ &$E(d, \frac{\pi}{8})$ & $E(d, \frac{\pi}{4})$\\[1.2ex]
\hline
Concentric & 0& 4.3540729& 4.3540841 &4.3540737\\
\hline
& 0.4 &5.00795	&5.00721 & 5.00645 \\
&0.5 & 5.57787	&5.57389& 5.56991\\
& 0.7& 11.32049& 10.42402 & 9.83979 \\
& 0.71 & 13.55325 	& 11.50058& 10.53196\\
& 0.715 &  {\bf 16.62217}	& 12.24159& 10.94060\\
& 0.73& NA & {\bf 17.80873}&12.53175 \\
& 0.78 & NA &NA &{\bf 63.88532} 
\\[0.1ex]
\hline
\end{tabular}
\end{center}
\label{table3}
\end{table}
\begin{table}[H]
\caption{Values of the energy functional $E(d, \theta)$ at different configurations for a pentagon obstacle}
\begin{center}
\begin{tabular}{|c| c| c|c|c|}
\hline
& & 
OFF ($\theta =0$)
&
INTERMEDIATE ($\theta = \frac{\pi}{10}$) 
& 
ON ( $\theta = \frac{\pi}{5}$ )
 \\ [1.5ex]
\hline
& $d$ & $E({d, 0})$ &$E(d, \frac{\pi}{10})$ & $E(d, \frac{\pi}{5})$\\[1.2ex]
\hline
Concentric & 0& 4.75447& 4.75448&4.75447\\
\hline
& 0.4 &5.5653& 5.56512	&5.56494  \\
& 0.5 & 6.30518 & 6.30378&	6.30239 \\
& 0.7& {\bf 19.40831} & {\bf 16.22683}  & 14.33736 \\
& 0.712&  NA & {\bf 23.41319} & 18.34214\\ 
& 0.72 & NA& NA& {\bf 19.94229}\\[0.1ex]
\hline
\end{tabular}
\end{center}
\label{table4}
\end{table}
In Tables \ref{table5} and \ref{table6}, we give the numerical data for the behaviour of the energy functional w.r.t. expansion and contraction of a square and a pentagonal obstacle, respectively. The value of $E$ does depend on $d$ and $\theta$ both. We fix $\theta =0$ for the square and $\theta =\frac{\pi}{10}$ for the pentagon.
\begin{table}[H]\caption{Values of the energy functional $E(\lambda,d)$ for a square obstacle when $\theta =0$} 
\begin{center}
\begin{tabular}{|c| c|c|c|}
\hline
$\lambda$ & d=0 & d=0.4 & d=0.7\\[0.1ex]
\hline
0.5  &2.96043	&3.23085	&4.41850\\[0.1ex]
1  &4.37408	&5.03007	&9.93242\\[0.1ex]
1.5  &6.07416	&7.61090	&NA\\[0.1ex]
2  &8.40577	&12.74641	&NA \\[0.1ex]
\hline
\end{tabular}
\end{center}
\label{table5}
\end{table}
\begin{table}[H]\caption{Values of the energy functional $E(\lambda, d)$ for a pentagonal obstacle when $\theta= \pi/10$}
\begin{center}
\begin{tabular}{|c|c|c|c|}
\hline
$\lambda$ & d=0 & d=0.4 & d=0.7\\[0.1ex]
\hline
0.5  &3.12893	&3.43447	&4.84344\\[0.1ex]
1  &4.76700	&5.58059	&16.44889\\[0.1ex]
1.5  &6.87274	&9.10172		&NA\\[0.1ex]
2  &10.01701	 &21.54845	&NA \\[0.1ex]
\hline
\end{tabular}
\end{center}
\label{table6}
\end{table}
In Table \ref{table7}, we give the numerical data for the behaviour of the energy functional when the obstacle is also circular. The proof of this, and of a more general problem is part of our upcoming paper. Please note that, in this case, $E$ is constant w.r.t. the parameter $\theta$. 
\begin{table}[H]\caption{Values of the energy functional $E$ at different locations of a circular obstacle}
\begin{center}
\begin{tabular}{|c|c|c|c|c|c|c|c|c|c|c|}
\hline
$d$ & 0 & 0.1 & 0.2 & 0.3 & 0.4 & 0.5 & 0.6 & 0.7\\[0.1ex]
\hline
$E$  &5.21871	&5.2671	&5.4221	&5.7192 &6.24655	&7.24692	&9.71217	&1679.1395\\[0.1ex]
\hline
\end{tabular}
\end{center}
\label{table7}
\end{table}
\section{Conclusions}\label{Conc}
In this paper, we considered an obstacle placement problem, where the planar obstacle is invariant under the action of a dihedral group. We considered an open disk in the Euclidean plane containing the obstacle such that the centers of the enclosing disk and the obstacle are non-concentric. This article is a follow up of the paper \cite{Anisa-Souvik}. The family of domains considered here is the same as the one considered in \cite{Anisa-Souvik} but the boundary value problem is different from that of \cite{Anisa-Souvik}. In \cite{Anisa-Souvik}, we had homogeneous boundary data while in the current paper we consider an inhomogeneous boundary data. Further, in the place of an eigenvalue problem considered in \cite{Anisa-Souvik}, we consider a Laplace equation here. We carry out the shape calculus analysis for the boundary value problem at hand and derive an expression for the Eulerian derivative of the energy functional. Once we have this expression, the proof for finding optimal configurartions w.r.t. the rotations of the obstacle about its fixed center follows more or less from \cite{Anisa-Souvik}. We prove this result for the case where the obstacle has an even order dihedral symmetry. For the case of odd order symmetry, we have partial results. We face exactly the same difficulties as in \cite{Anisa-Souvik} in characterizing the optimal configurations for the odd order case completely. 
We thereby formulate conjectures about such configurations based on numerical evidence. 

We further characterize the global maximizing and the global minimizing configurations w.r.t. the rotations of the obstacle about its center as well as the translations of the obstacle within the disk. For the odd order case, we identify the global minimizing configuration and partially identify the global maximizing configurations as opposed to the even order case, where we completely characterize them. 

We also prove that the energy functional strictly increases as the obstacle $P$ expands inside the disk $B$ and strictly decreases as the obstacle $P$ contracts inside $B$. Further, the energy functional approaches its minimum value as the obstacle $P$ shrinks to a point. 
It's worth mentioning here that the BVP  (\ref{laplace_equation_lc}) does not admit any solution when $P$ shrinks to the center of $B$. 
Also, a harmonic function on any planar disk (without any hole or puncture) with constant boundary condition always has zero energy. 

Finally, we provide some numerical evidences to support our theoretical findings and the conjectures.

It's not difficult to see that the energy functional increases as $|M|$, the absolute value of the boundary data, increases. In fact, the energy functional scales by a factor of $\alpha^2$ as the value of the boundary data $g$ scales from $M$ to $\alpha \, M$, $\alpha \in \mathbb{R}$.
\section{Acknowledgments} Thanks to Dr. Apala Majumdar, University of Bath, UK for introducing the author to the obstacle placement problems for Liquid Crystals, and for suggesting the boundary value problem which has potential applications in some liquid crystal context. We thank Dr. Rajesh Mahadevan, Universidad de Concepcion, Chile for suggesting a reference for the shape calculus part. We thank Dr. Souvik Roy, University of Texas, Arlington and Ashok Kumar K., IIT Madras for their help in providing the numerical data. 
We also thank University of Bath and Universidad de Concepcion for their hospitality.

\end{document}